\documentclass[a4paper,11pt,reqno]{amsart}

\title{Infinite-dimensional Polish groups and Property (T)}

\author{Tom\'as Ibarluc{\'\i}a}
\thanks{Research partially supported by the ANR contract AGRUME (ANR-17-CE40-0026).}
\address{
  Universit{\'e} de Paris\\
  CNRS, Institut de Math{\'e}matiques de Jussieu--Paris Rive Gauche\\
  F-75013 Paris\\
  France}
\urladdr{https://webusers.imj-prg.fr/~tomas.ibarlucia}

\usepackage{amsmath, amssymb, amsthm, amscd, enumitem, mathtools}

\usepackage{tikz}

\usepackage[T1]{fontenc}
\usepackage{kpfonts}

\usepackage{geometry}
\usepackage{stmaryrd}

\makeatletter
\g@addto@macro\bfseries{\boldmath}
\makeatother

\def\mbN{\mathbb{N}}
\def\mbZ{\mathbb{Z}}
\def\mbQ{\mathbb{Q}}
\def\mbR{\mathbb{R}}

\def\mbC{\mathbb{C}}
\def\mbF{\mathbb{F}}
\def\mbI{\mathbb{I}}
\def\mbP{\mathbb{P}}
\def\mbL{\mathbb{L}}
\def\mbT{\mathbb{T}}

\def\wM{\widehat M}

\DeclareMathOperator{\Sym}{Sym}
\DeclareMathOperator{\Unit}{U}
\DeclareMathOperator{\GL}{GL}
\DeclareMathOperator{\Aut}{Aut}
\DeclareMathOperator{\End}{End}

\DeclareMathOperator{\Homeo}{Homeo}

\DeclareMathOperator{\acl}{acl}
\DeclareMathOperator{\dcl}{dcl}

\DeclareMathOperator{\tp}{tp}

\newcommand{\eq}{\mathrm{eq}}
\newcommand{\meq}{\mathrm{meq}}

\theoremstyle{plain}        \newtheorem{fact}{Fact}[section]
\theoremstyle{plain}        \newtheorem{theorem}[fact]{Theorem}
\theoremstyle{plain}        \newtheorem*{theorem*}{Theorem}
\theoremstyle{plain}        \newtheorem{lem}[fact]{Lemma}
\theoremstyle{plain}        \newtheorem{prop}[fact]{Proposition}
\theoremstyle{plain}        \newtheorem{cor}[fact]{Corollary}
\theoremstyle{plain}        \newtheorem*{cor*}{Corollary}
\theoremstyle{definition}   \newtheorem{rem}[fact]{Remark} 
\theoremstyle{definition}   \newtheorem{defin}[fact]{Definition}
\theoremstyle{definition}   
\theoremstyle{definition}   
\theoremstyle{definition}   
\theoremstyle{definition}   \newtheorem*{question*}{Question}
\theoremstyle{definition}   

\numberwithin{equation}{section}

\usepackage[hidelinks]{hyperref}

\def\Ind#1#2{#1\setbox0=\hbox{$#1x$}\kern\wd0\hbox to 0pt{\hss$#1\mid$\hss}
\lower.9\ht0\hbox to 0pt{\hss$#1\smile$\hss}\kern\wd0}
\def\ind{\mathop{\mathpalette\Ind{}}}
\def\Notind#1#2{#1\setbox0=\hbox{$#1x$}\kern\wd0\hbox to 0pt{\mathchardef
\nn="3236\hss$#1\nn$\kern1.4\wd0\hss}\hbox to 0pt{\hss$#1\mid$\hss}\lower.9\ht0
\hbox to 0pt{\hss$#1\smile$\hss}\kern\wd0}

\usepackage{graphicx}
\makeatletter
\newcommand{\bigperp}{
  \mathop{\mathpalette\bigp@rp\relax}
  \displaylimits
}
\newcommand{\bigp@rp}[2]{
  \vcenter{
    \m@th\hbox{\scalebox{\ifx#1\displaystyle1.7\else1.4\fi}{$#1\perp$}}
  }
}
\makeatother

\usepackage{xcolor}

\newcommand{\ov}[1]{\overline{#1}}
\newcommand{\dd}[2]{\textrm{d}_{\! #1}\! #2}

\DeclarePairedDelimiter\set{\{}{\}}
\DeclarePairedDelimiter\norm{\|}{\|}
\DeclarePairedDelimiter\abs{|}{|}
\DeclarePairedDelimiter\bigabs{\big|}{\big|}

\DeclarePairedDelimiter\ip{\langle}{\rangle}

\DeclarePairedDelimiter\gen{\langle}{\rangle}

\newcommand{\actson}{\curvearrowright}
\newcommand{\inv}{^{-1}}

\begin{document}

\begin{abstract}
We show that all groups of a distinguished class of \guillemotleft large\guillemotright\ topological groups, that of Roelcke precompact Polish groups, have Kazhdan's Property (T). This answers a question of Tsankov and generalizes previous results by Bekka (for the infinite-dimensional unitary group) and by Evans and Tsankov (for oligomorphic groups). Further examples include the group $\Aut(\mu)$ of measure-preserving transformations of the unit interval and the group $\Aut^*(\mu)$ of non-singular transformations of the unit interval. 

More precisely, we prove that the smallest cocompact normal subgroup $G^\circ$ of any given non-compact Roelcke precompact Polish group $G$ has a free subgroup $F\leq G^\circ$ of rank two with the following property: every unitary representation of $G^\circ$ without invariant unit vectors restricts to a multiple of the left-regular representation of $F$. The proof is model-theoretic and does not rely on results of classification of unitary representations. Its main ingredient is the construction, for any $\aleph_0$-categorical metric structure, of an action of a free group on a system of elementary substructures with suitable independence conditions.
\end{abstract}

\maketitle
\setcounter{tocdepth}{1}
\tableofcontents

\section*{Introduction}

This paper is concerned with the study of unitary representations and Kazhdan's Property~(T) outside the realm of locally compact groups, which has received increased attention in recent years. Although admitting no invariant measures and being devoid of regular representations, many Polish non-locally compact groups have interesting unitary representations, and call for the development of new techniques for their understanding. In return for the lost tools, new phenomena appear, such as families of groups in which Property~(T) and amenability are no longer contradictory but often coincidental.

\subsection*{Background} The first interesting example of an \guillemotleft infinite-dimensional\guillemotright\ group with Property~(T) was given by Shalom \cite{shaLoop}, who exhibited finite Kazhdan sets for the loop groups $\text{L}(\text{SL}_n(\mbC))$, $n\geq 3$. (See Section \ref{subsec:Prop-(T)} for the basic definitions concerning Property~(T).) A few years later, Bekka \cite{bekkaUnitary} proved that the unitary group $\text{U}(\ell^2)$ has Property~(T). For this he used the classification of the unitary representations of $\Unit(\ell^2)$ obtained by Kirillov and Ol'shanski, and showed that the generators of any shift action $\mbF\actson\ell^2(\mbF)$ by a non-abelian free group of finite rank form a Kazhdan set for $\Unit(\ell^2(\mbF))$.

Within the framework of infinite-dimensional \emph{permutation groups}, a milestone was set by Tsankov \cite{tsaUnitary} with the classification of the unitary representations of oligomorphic groups. As a consequence of his result, he was able to prove Property~(T) for a number of significant examples including the infinite symmetric group $S_\infty$, the group $\Aut(\mbQ,<)$ of order-preserving bijections of the rationals and the group $\Homeo(2^\omega)$ of homeomorphisms of the Cantor space. He showed, for instance, that the permutation action $\mbF\actson\mbF$ induces a Kazhdan set for $S_\infty=\Sym(\mbF)$, very much like in the case of the unitary group. Subsequently, Evans and Tsankov \cite{evatsaFree} succeeded in generalizing the analysis of \cite{tsaUnitary} and established Property~(T) for all oligomorphic groups, as well as for a related larger class of permutation groups.

In a recent work, Pestov \cite{pestovVersus} showed that amenability and (strong) Property~(T) remain contradictory within the class of unitarily representable SIN groups, thus providing several \emph{non}-examples of strong Property~(T) in the non-locally compact setting. In particular, if $G$ is a non-trivial, compact, metrizable group, the Polish group $L^0([0,1],G)$ of random elements of $G$ does not admit a finite Kazhdan set. In the case of the circle $G=\mbT$, Solecki \cite{solUnitary} classified the unitary representations of the corresponding randomized group. Inspired by Solecki's work, Pestov proved moreover that $L^0([0,1],\mbT)$ does not admit any compact Kazhdan set. The last section of \cite{pestovVersus} provides a nice summary of the current state of knowledge concerning Kazhdan's~(T) and related properties for several important examples of infinite-dimensional topological groups.

The unitary group studied by Bekka and the permutation groups covered by the works of Evans and Tsankov have something in common. They are all \emph{Roelcke precompact} topological groups. This means that their completion with respect to the \emph{Roelcke} (or \emph{lower}) uniformity is compact or, equivalently, that for every open set $U\subseteq G$ there exists a finite set $F\subseteq G$ such that $G=UFU$. Other examples in this family are the group $\Aut(\mu)$ of measure-preserving transformations of the Lebesgue space $([0,1],\mu)$, the group $\Aut^*(\mu)$ of measure-class-preserving transformations of the same space, the semi-direct product $L^0([0,1],\mbT)\rtimes\Aut(\mu)$, or the group $\Homeo_+([0,1])$ of increasing homeomorphisms of the interval. Question~(1) at the end of \cite{tsaUnitary} asked whether every Polish Roelcke precompact (\emph{PRP}) group has Property~(T).

Later, Ben Yaacov and Tsankov \cite{bentsa} realized that PRP groups are exactly those topological groups that appear as automorphism groups of \emph{$\aleph_0$-categorical structures} in the sense of continuous logic. These are separable structures that are characterized up to isomorphism by their first-order properties, such as the Hilbert space $\ell^2$. Oligomorphic groups correspond to the particular case given by classical (i.e., 2-valued) logic, as has been known for long ---a fact that had been exploited in \cite{tsaUnitary}. The new characterization of PRP groups opened the door for novel interactions with model-theory. A precise dictionary between several topological-dynamic features of PRP groups and model-theoretic properties of the associated structures arose from the works \cite{bentsa,iba16,bit18}, along with several applications.

\subsection*{Results and examples} In this paper we use model-theoretic methods to prove that PRP groups are Kazhdan, thus answering Tsankov's question. Moreover, we show that this holds in a particularly strong form, improving also over the previously known cases.

\begin{theorem*}
Every Polish Roelcke precompact group $G$ has Property~(T). Moreover, up to passing to a cocompact normal subgroup (and assuming $G$ is non-compact), $G$ has a freely two-generated subgroup $F$ such that every unitary representation of $G$ with no invariant unit vectors restricts to a multiple of the left-regular representation of $F$.
\end{theorem*}

In particular, if $G$ is a PRP group with no compact quotients, then $G$ has a Kazhdan set with two elements. We can actually deduce the following, which in particular confirms, for PRP groups, a suspicion of Bekka \cite[p.~512]{bekkaUnitary} that was disproved by Pestov \cite[\textsection 8]{pestovVersus} for arbitrary topological groups.

\begin{cor*}
A Roelcke precompact Polish group has a finite Kazhdan set if and only if its Bohr compactification does.
\end{cor*}

Two conspicuous concrete cases of our theorem are the groups $\Aut(\mu)$ and $\Aut^*(\mu)$ mentioned above. Pestov \cite[\textsection 9]{pestovVersus} lists the question of whether these groups are Kazhdan as open. In the case of $\Aut(\mu)$, however, it seems likely that Property~(T) could be deduced from the work of Neretin \cite{nerBistochastic}\cite[\textsection 8.4]{nerBook} on the representations of this group, in a similar fashion as Bekka's proof for $U(\ell^2)$ from the works of Kirillov and Ol'shanski (although the works \cite{nerBistochastic,nerBook} only contain a description of the \emph{irreducible} representations of $\Aut(\mu)$, and an argument has to be added in order to complete the classification of all unitary representations). In the case of $\Aut^*(\mu)$, on the other hand, I believe that the result is indeed new. In both cases we are able to give explicit finite Kazhdan sets (see Section~\ref{sec:special-cases}).

Another interesting family of examples is related to the randomized groups $L^0([0,1],G)$. As we said before, these groups generally fail to have Property~(T). However, if $G$ is a PRP group, then the semidirect product $L^0([0,1],G)\rtimes\Aut(\mu)$ is a PRP group as well, as shown in \cite{ibaRando}. Hence the semidirect product enjoys Property~(T), and it even has a finite Kazhdan set (which cannot be contained in $\Aut(\mu)$, if $G$ has non-trivial representations). In particular, and in contrast to Pestov's result, all groups of the form $L^0([0,1],G)\rtimes\Aut(\mu)$ for compact metrizable $G$ have strong Property~(T).

\begin{question*}
Let $G$ be a Polish, Kazhdan group. Is $L^0([0,1],G)\rtimes\Aut(\mu)$ also Kazhdan?
\end{question*}

A similar example is given by the semi-direct product $L^0([0,1],\mbT)\rtimes\Aut^*(\mu)$. The latter can be identified with the group of linear isometries of the Banach space $L^1([0,1])$, which is an $\aleph_0$-categorical structure. Hence $L^0([0,1],\mbT)\rtimes\Aut^*(\mu)$ is a PRP group and also has Property~(T).

\subsection*{Main ideas} Let us comment on the proof of our theorem and discuss some of its ingredients. A feature of our proof is that unlike the previous results, it does not rely on any classification result of unitary representations: such a classification is not currently available for general PRP groups.

In the case of oligomorphic groups, which are permutation groups, one has natural substitutes for the regular and quasi-regular representations of the locally compact setting. Indeed, when $M$ is a countable discrete structure and $G=\Aut(M)$, one has at hand the representation $G\actson\ell^2(M)$ as well as all representations $G\actson\ell^2(S)$ where $S$ is an imaginary sort of $M$ (i.e., the quotient of $M^n$ by a definable equivalence relation). The results of Tsankov \cite{tsaUnitary} actually show that if $M$ is countably categorical, then \emph{every separable unitary representation of $G$ is a subrepresentation of $G\actson\ell^2(M^\eq)$} (where $M^\eq$ is the structure regrouping all imaginary sorts of $M$). Whence the strategy applied in \cite{tsaUnitary}: if the structure $M$ has weak elimination of imaginaries, to prove Property~(T) it suffices to find an action of a non-abelian free group on $M$ that is \emph{free} on the non-algebraic part of $M$. This is easy to produce in many concrete examples: see \cite[Theorem~6.7]{tsaUnitary}. The main contribution of the subsequent work \cite{evatsaFree} by Evans and Tsankov is the construction of a free action $\mbF\actson M^\eq\setminus\acl(\emptyset)$ of a two-generated free group for \emph{any} classical $\aleph_0$-categorical $M$. The construction is achieved via a back-and-forth argument based on Neumann's Lemma (\cite[Lemma~2.1]{evatsaFree}).

The situation is similar for the unitary group, as we said before: a Kazhdan set is given by the action $\mbF\actson\ell^2(\mbF)$.

For an arbitrary PRP group, on the other hand, there might be no natural representations at hand. There are in fact PRP groups that do not have any non-trivial unitary representations, as first shown by Megrelishvili \cite{meg01} for the case of $\Homeo_+([0,1])$. Nevertheless, when they exist, the unitary representations induce imaginary sorts of the associated $\aleph_0$-cate\-gorical metric structure. More precisely:

\smallskip
\begin{itemize}[leftmargin=15pt]
\item[]\emph{Let $G=\Aut(M)$. If $G\actson H$ is a unitary representation and $\xi\in H$ is any vector, then the closed orbit $O=\ov{G\xi}$ is a metric imaginary sort of~$M$.}
\end{itemize}
\smallskip
Thereby all the information about the unitary representations of $G$ is again coded in $M^\meq$ (the metric analogue of the $M^\eq$ construction). This is a much weaker (other than basic and well-known) statement than the one about the oligomorphic case stated above, but we will use it to transfer certain model-theoretic configurations to the representations.

As in \cite{evatsaFree}, the crucial point of our proof is the construction of a \guillemotleft very free\guillemotright\ action $\mbF\actson M$ of a two-generated free group on a separably categorical structure. However, instead of a usual \guillemotleft internal\guillemotright\ back-and-forth construction based on subsets of $M$, we will perform an \guillemotleft external\guillemotright\ back-and-forth\footnote{This is an idea that I learned (for $\mbZ$ instead of $\mbF$, and using a different notion of independence) from a talk by Pierre Simon, in which he drew the picture that we reproduce here. I believe he was explaining Theorem~3.12 of \cite{kapsimTopRank}, although there is unfortunately no drawing in the paper.} by piling up isomorphic copies of $M$. At the same time, we replace Neumann's Lemma with (local) stability theory. Let us illustrate the construction here by showing how to build the following (under the assumption that $G$ has no non-trivial compact quotients):

\smallskip
\begin{itemize}[leftmargin=15pt]
\item[]\emph{There is an action $\mbZ\actson M$ such that every unitary representation of $G$ with no invariant unit vectors restricts to a multiple of the left-regular representation $\mbZ\actson\ell^2(\mbZ)$.}
\end{itemize}
\smallskip

The main tool for the construction is the \emph{stable independence relation}, which we denote by~$\ind$, and which makes sense in any (possibly unstable) metric theory. In \textsection\ref{prem:Stable-independence} we review the main properties of $\ind$, which condense some well-known facts of the theory of local stability. For some reason, this \guillemotleft semi-global\guillemotright\ notion of independence is hardly ever presented or considered in this way in the model-theoretic literature. It can however be very useful in the context of applications, as has already been shown in the works \cite{bit18} and \cite{ibatsaStrong}.

The construction begins with an isomorphic copy $M_0\simeq M$ inside some large elementary extension of $M$. We then pick an independent copy $M_1\ind M_0$ and a common elementary extension $M_{01}\succeq M_0,M_1$. This gives the smaller \guillemotleft cairn\guillemotright\ at the center of the picture below.

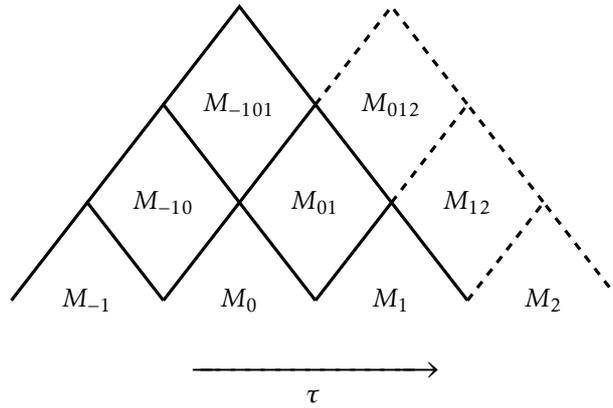
\vspace{15pt}

\begin{center}
\begin{tikzpicture}[xscale=2,yscale=1.5*sqrt(.75),xslant=-.5]

\tikzset{coordlabels/.style={font=\small,very thin,rectangle,rounded corners=6pt}}

\foreach \y in {-1,...,1}
\foreach \x in {\y,...,1}
\draw[very thick,shift={(\x,\y)}] (0,0)--(1,1)--(1,0);

\foreach \y in {-1,...,1}
\draw[very thick,dashed,shift={(2,\y)}] (0,0)--(1,1)--(1,0);

\foreach \x in {-1,...,2}
\draw ( \x+.5, -1) node [coordlabels]{$M_{\x}$};

\draw ( 0+.5, 0) node [coordlabels]{$M_{-10}$};
\draw ( 1+.5, 0) node [coordlabels]{$M_{01}$};
\draw ( 2+.5, 0) node [coordlabels]{$M_{12}$};

\draw ( 1+.5, 1) node [coordlabels]{$M_{-101}$};
\draw ( 2+.5, 1) node [coordlabels]{$M_{012}$};

\draw ( .6, -1.8) node [coordlabels]{{\Large $\xrightarrow[\tau]{\hspace{3cm}}$}};

\end{tikzpicture}
\end{center}
We now perform a step backwards: we choose models $M_{-10}$, $M_{-1}$ such that $M_{-10}\ind_{M_0} M_{01}$ and $M_{-10}M_{-1}M_0 \simeq M_{01}M_0 M_1$ (i.e., there is an isomorphism sending $M_{-10}$ to $M_{01}$, $M_{-1}$ to $M_0$ and $M_0$ to $M_1$). We complete the new cairn by choosing any common elementary extension $M_{-101}\succeq M_{-10},M_{01}$. The construction continues forward by choosing $M_{012}$ independent from $M_{-101}$ over $M_{01}$, together with submodels $M_{12}$ and $M_2$ such that the cairns dominated by $M_{-101}$ and $M_{012}$ are isomorphic. And so on.

At the end we get a directed family $\set{M_I}$ of elementary extensions of copies of $M$, indexed by the finite intervals of $\mbZ$. By $\aleph_0$-categoricity, we may identify its union $\ov{\bigcup_I M_I}$ with $M$. Then, by construction, the shift $M_I\mapsto M_{I+1}$ defines an automorphism $\tau\in G=\Aut(M)$. The group $Z\leq G$ generated by $\tau$ has the desired property.

Indeed, suppose $G\actson H$ is a unitary representation with no invariant unit vectors. We may assume $H$ has a cyclic vector $\xi$ (i.e., such that $H=\ov{\gen{G\xi}}$) and consider the closed orbit $O=\ov{G\xi}$ as earlier. Now, as $O$ is an imaginary sort of $M$, the submodels $M_I\preceq M$ induce substructures $O_I\preceq O$. Consider the corresponding generated subspaces $H_I=\ov{\gen{O_I}}\subseteq H$, as well as $H_\emptyset=\set{0}$. Then a key lemma will show that the independence conditions on the $M_I$ carry over to the subspaces $H_I$:

\smallskip
\begin{itemize}[leftmargin=15pt]
\item[]\emph{If $I,J\subseteq\mbZ$ are intervals and $K=I\cap J$, then $H_I\bigperp_{H_K}H_J$.}
\end{itemize}
\smallskip
Here, $\bigperp$ denotes the orthogonality relation. As a result, if we define $E_n=\ov{\gen{H_I:|I|=n}}$ and, for each interval $I$ of length $n$, $\widetilde H_I=H_I\ominus E_{n-1}$ (the orthogonal projection of $H_I$ to the orthogonal complement of $E_{n-1}$), we obtain that
$$H=\bigoplus_{n\in\mbN}\bigoplus_{|I|=n}\widetilde H_I.$$ Since $\tau$ permutes the subspaces $\widetilde H_I$ within each smaller direct sum, we see that $H$ splits as a multiple of the left-regular representation of $Z$, as desired.

Certainly, in order to prove Property~(T) we have to replace $\mbZ$ in the previous construction by a non-abelian free group $\mbF$. A crucial point then is the analysis of the \guillemotleft intervals\guillemotright\ $I\subseteq\mbF$ that appear in the construction, to ensure that they behave like the intervals of $\mbZ$ in a few key aspects. This analysis is carried out in Section~\ref{sInt}, which is independent of the rest. Then, in Section~\ref{sCai}, we build the cairn of models over a free group. The main theorem is proved subsequently in Section~\ref{sec:PropT}.

Finally, when the structure $M$ is stable and has the property that the structure generated by any two elementary substructures is again an elementary substructure, the construction described above simplifies considerably, as the reader may have already noticed. This allows for an explicit description of Kazhdan sets in some concrete examples. We do this in Section~\ref{sec:special-cases}.

\subsection*{Acknowledgments} I would like to thank Todor Tsankov for valuable discussions.

\noindent\hrulefill

\section{Preliminaries on metric imaginaries and stable independence}

Throughout the paper we will assume some familiarity with the model theory of metric structures as presented in \cite{bbhu08} or \cite{benusv10}. Nevertheless, we review in this section the main notions and facts that we will use.

Enthusiastic readers not so familiar with model theory may find some relief in \textsection\ref{apartado:a0-cat} below, where we recall some dynamical reformulations.

\subsection{Metric imaginaries}
Our main reference here is Ben Yaacov's work \cite[\textsection 1]{benCannot}, although we take a (formally) more general point of view in that we add imaginary sorts for definable subsets and quotients thereof. Earlier expositions on imaginaries in continuous logic can be found in \cite[\textsection 11]{bbhu08} and \cite[\textsection 5]{benusv10}.

Let $M$ be a structure in a possibly multi-sorted metric language. We recall that each sort of $M$ is a complete metric space of bounded diameter. The letters $x$, $y$, etc.\ will denote tuples of variables, usually finite or countably infinite. Say $x=(x_i)$ is countable. Each individual variable $x_i$ is attached to a particular sort $(S_i,d_i)$. Then by $M^x$ we will denote the corresponding \emph{product sort} of~$M$, that is, the product $\prod S_i$ endowed with a complete, bounded metric $d$ defined in terms of the metrics $d_i$ and compatible with the product uniformity (for instance, we can set $d=\max d_i$ in the case of finitely many factors, and $d=\sum 2^{-n_i}d_i$ in the countably infinite case, provided that $\sum 2^{-n_i}\text{diam}(S_i)<\infty$). An element of $M^x$ is an \emph{$x$-tuple} of $M$.

We will make no difference between formulas and definable predicates, which may thus depend on a countably infinite number of variables. By \emph{definable} we will always mean $\emptyset$-definable, that is, without parameters; if we allow parameters from a subset $B\subseteq M$, we will write \emph{$B$-definable}.

Suppose we are given a definable set $D\subseteq M^x$ and a definable pseudo-metric on $D$, say $\rho\colon D\to \mbR_{\geq 0}$. Then $\rho$ induces a metric on the classes of the equivalence relation $\sim_\rho$ that identifies the pairs of $x$-tuples in $D$ at $\rho$-distance 0. We let $M_\rho$ be the completion of the quotient $D/\sim_\rho$ with respect to the distance~$\rho$. The resulting (complete, bounded) metric space $M_\rho$ is called a \emph{metric imaginary sort} of $M$, and the elements of $M_\rho$ are \emph{metric imaginaries} of~$M$.

We denote by $M^\meq$ the collection of all metric imaginary sorts of $M$. This contains $M$ as a particular case (namely, $M=M_d$ where $d$ is the metric of $M$). The collection $M^\meq$ can then be turned into a multi-sorted metric structure, of which $M$ is a reduct, essentially by adding predicates to render the canonical inclusions $D\to M^x$ and projections $D\to M_\rho$ definable; see \cite{benCannot} for some details. For that matter, we may go further and add symbols for all the resulting definable predicates and functions. The language of $M^\meq$ is in any case huge, for it already includes an uncountable number of sorts. We should make clear that this construction is just a convenient artifice to handle the imaginaries in the models that we will consider; all structures we are truly interested in are separable in countable languages, and all arguments in the paper can be rewritten so as to consider only structures of this kind.

The structure $M^\meq$ is completely determined by $M$ in that, for instance, every elementary extension $M\preceq N$ lifts \emph{naturally} to an elementary extension $M^\meq\preceq N^\meq$. In particular, the automorphism groups $\Aut(M)$ and $\Aut(M^\meq)$ can be identified.

A structure $N$ (in a possibly different language than that of $M$) is \emph{interpretable in $M$} if it is isomorphic to a reduct of $M^\meq$, after a suitable identification of the languages of $N$ and of the reduct. It is \emph{bi-interpretable with $M$} if the associated expansions $N^\meq$ and $M^\meq$ are isomorphic, after a suitable identification of their languages. For more precise definitions concerning interpretations, see Ben Yaacov and Ka{\"{\i}}chouh's work \cite{benkai13}. The structures $M$ and $M^\meq$ are bi-interpretable.

\subsection{Types and algebraic elements} If $a$ is an $x$-tuple of elements of $M$ and $B\subseteq M$ is any subset, $\tp(a/B)$ denotes the \emph{type of $a$ over $B$}, that is, the function that maps each $B$-definable predicate $\phi(x)$ to the truth value $\phi(a)\in\mbC$ as calculated in $M$. If $a'$ is any other $x$-tuple, we write as usual $a'\equiv_B a$ to indicate that $a'$ and $a$ have the same type over $B$. If $B$ is empty, we write simply $a'\equiv a$.

An element $a\in M$ is \emph{definable over $B$} if in every elementary extension $N\succeq M$, the only element $a'\in N$ with $a'\equiv_B a$ is $a$ itself. Similarly, $a\in M$ is \emph{algebraic over $B$} if the set $\set{a'\in N:a'\equiv_B a}$ is compact in every elementary extension $N\succeq M$. If $a$ is definable over $\emptyset$ (respectively, algebraic over $\emptyset$), we say simply that it is \emph{definable} (respectively, \emph{algebraic}).

Finally, as is usual practice, we denote by $\dcl(B)$ the set of \emph{imaginary} elements of $M$ that are definable over $B$, that is, the set of all elements from $M^\meq$ definable over $B$. It is called the \emph{definable closure of $B$}. Similarly, the \emph{algebraic closure of $B$}, $\acl(B)$, is the set of imaginary elements of $M$ algebraic over $B$. In particular, we have $\dcl(M)=\acl(M)=M^\meq$. The \emph{real} algebraic closure of $B$ is the set $\acl(B)\cap M$.

\subsection{The $\aleph_0$-categorical case}\label{apartado:a0-cat}
Everything becomes more concrete if the structure $M$ is $\aleph_0$-cate\-gorical. We recall that if the language of $M$ is countable, $M$ is said $\aleph_0$-categorical if its first-order theory admits a \emph{unique separable model} up to isomorphism, namely $M$ (in particular, $M$ is separable). If the language underlying a given structure $M^*$ is not countable but $M^*$ is bi-interpretable with some $\aleph_0$-categorical structure $M$ in a countable language, we will say that $M^*$ is $\aleph_0$-categorical as well (in particular, each sort of $M^*$ is separable). Let us also recall that the automorphism group of an $\aleph_0$-categorical structure, endowed with the topology of pointwise convergence, is always a Polish Roelcke precompact group, and that every such group can be presented in this way (moreover, the structure can be chosen to be one-sorted, for instance with underlying space given by the completion of the group with respect to a left-invariant metric); see \cite[\textsection 2]{bentsa}.

As we were saying, in the $\aleph_0$-categorical setting most model-theoretic notions boil down to (topological) group-theoretic ones. Let us list some basic entries of this dictionary. Below, $M$ denotes an $\aleph_0$-categorical structure and $G=\Aut(M)$ is its automorphism group with the pointwise convergence topology. The product sorts $M^x$ are given the product topology and $G$ acts on them by the diagonal action. We have:
\begin{itemize}[leftmargin=15pt]
\item A subset $D\subseteq M^x$ is definable iff it is $G$-invariant and closed.
\item A function $\phi\colon M^x\to\mbC$ is a definable predicate iff it is $G$-invariant and continuous.
\item Two tuples $a,b\in M^x$ have same type iff their closed orbits coincide: $a\equiv b\iff\ov{Ga}=\ov{Gb}$.
\item An element $a\in M$ is definable iff it is $G$-invariant.
\item An element $a\in M$ is algebraic iff the closed orbit $\ov{Ga}$ is compact.
\item A one-sorted structure $N$ is interpretable in $M$ iff there is a continuous action $G\actson^\sigma N$ by (isometric) automorphisms of $N$ such that the space of closed orbits $N\sslash \sigma(G)$ is compact.
\end{itemize}
For the last point see \cite{benkai13}. We single out and sketch the proof of the following particular case, for which we may as well refer to \cite[Fact~1.12]{benCannot}.

\begin{lem}\label{lem:G-actson-O}
Let $M$ be an $\aleph_0$-categorical structure and $G$ its automorphism group. Suppose $G\actson^\sigma \Xi$ is a continuous isometric action of $G$ on a complete metric space $\Xi$. Let $\xi\in \Xi$ and let $O=\ov{\sigma(G)\xi}$ be its closed orbit. Then $O$ is isometric to an imaginary sort of $M$, and the action $G\actson^\sigma O$ coincides with the restriction of the action $G\actson M^\meq$.
\end{lem}
\begin{proof}
Fix some countable tuple $a\in M^x$ whose definable closure is the whole structure. Then the map $ga\in Ga\mapsto \sigma_g\xi\in \sigma(G)\xi$ is well defined, and it extends to a continuous $G$-equivariant map $\ov{Ga}\to O$. In turn, composing with the metric of $\Xi$ we obtain a $G$-invariant pseudo-metric $\rho\colon \ov{Ga}\times \ov{Ga}\to \mbR_{\geq 0}$. By $\aleph_0$-categoricity, $\rho$ is a definable pseudo-metric on the definable set $D=\ov{Ga}$, and we can identify $O$ with $M_\rho$.
\end{proof}

Another important fact about $\aleph_0$-categoricity is that it is preserved after naming the algebraic closure of the empty set. More precisely, let $M$ and $G$ be as above. We consider $M^\meq$ and we enrich its language with new constants for every element in $\acl(\emptyset)$. We denote the resulting expansion of $M^\meq$ by $M^\circ$. The point of this construction is that $M^\circ$ now satisfies $\dcl^{M^\circ}(\emptyset)=\acl^{M^\circ}(\emptyset)$, that is, every algebraic element of $M^\circ$ is definable.

We denote the automorphism group of $M^\circ$ by $G^\circ$. In other words,
$$G^\circ=\Aut(M^\meq/\acl(\emptyset))\leq \Aut(M^\meq)=G,$$
i.e., $G^\circ$ is the group of automorphisms of $M$ that fix all algebraic imaginary elements of $M$.

\begin{prop}\label{prop:Mcirc-is-cat}
If $M$ is $\aleph_0$-categorical, then so is $M^\circ$. Hence $G^\circ$ is Polish and Roelcke precompact.
\end{prop}
\begin{proof}
See \cite[Proposition~1.15]{benCannot}.
\end{proof}

Now let $bG=\Aut_M(\acl(\emptyset))$ denote the group of partial elementary self-embeddings of $M^\meq$ whose domain is the algebraic closure of the empty set (note that any such partial self-embedding is indeed an isometry onto $\acl(\emptyset)$). Endowed with the topology of pointwise convergence, $bG$ becomes a \emph{compact} metrizable group, and the restriction map $G\to bG$ a continuous homomorphism. On the other hand, it is an immediate consequence of the $\aleph_0$-categoricity of $M^\circ$ that the restriction map $G\to bG$ is surjective. In other words, $bG=\set{g|_{\acl(\emptyset)}:g\in\Aut(M^\meq)}$.

\begin{cor}\label{cor:Bohr} If $M$ is $\aleph_0$-categorical, then we have the exact sequence of Polish groups:
$$1\to G^\circ\to G \to bG \to 1.$$
Moreover, $bG$ is the Bohr compactification of $G$ and $G^\circ$ is the smallest cocompact normal subgroup of~$G$.
\end{cor}
\begin{proof}
The moreover part also follows from \cite{benCannot}, but let us sketch the proof. It suffices to show that if $\pi\colon G\to K$ is a continuous homomorphism with dense image into a compact group $K$, then $G^\circ$ is in the kernel of $\pi$. Now, it is a general fact that every such $\pi$ factors through the \emph{Roelcke compactification} of $G$, and the latter is metrizable when $G$ is a Roelcke precompact Polish group. It follows that $K$ is a compact metrizable group, and hence admits an invariant metric. Thus $\pi$ induces a continuous \emph{isometric} action $G\actson K$, and the $G$-orbit of $1_K$ is dense in $K$ by hypothesis. Hence, by Lemma~\ref{lem:G-actson-O}, $K$ is a metric imaginary sort of $M$ and the action $G\actson K$ is given by the action of $G$ by automorphisms of $M$. Since $K$ is compact, every element of $K$ is an \emph{algebraic} metric imaginary of $M$, and thus is fixed by $G^\circ$. This shows that $G^\circ\subseteq\ker(\pi)$.
\end{proof}

\subsection{Stable independence}\label{prem:Stable-independence}

Shelah's fundamental theory of stability was recast in the setting of continuous logic in the work \cite{benusv10} by Ben Yaacov and Usvyatsov. See also \cite{benGro} for a functional analysis approach based on results of Grothendieck predating those of Shelah.

In this subsection we fix an ambient structure $\wM$, which is a very saturated model of an arbitrary (possibly unstable) metric first-order theory. All sets and tuples we consider come from $\wM$ (or $\wM^\meq$) and are small relative to the order of saturation of $\wM$; by a \emph{submodel} we mean a small elementary substructure of $\wM$.

A formula $\phi(x,y)$ is said \emph{stable} if one cannot find $\epsilon>0$ and infinite sequences $a_i\in\wM^x$, $b_j\in\wM^y$ such that $|\phi(a_i,b_j)-\phi(a_j,b_i)|\geq\epsilon$ for every $i,j\in\mbN$, $i\neq j$. An equivalent definition is as follows. Given an $x$-tuple $a$ and a submodel $N$, we associate to the type $p=\tp(a/N)$ its \emph{$\phi$-definition}, $\dd{p}{\phi}\colon N^y\to\mbC$, given by
$$\dd{p}{\phi}(b)\coloneqq\phi(a,b).$$
Then the formula $\phi(x,y)$ is stable iff for every $x$-tuple $a$ and every submodel $N$, the $\phi$-definition of $\tp(a/N)$ is an $N$-definable predicate. The theory of $\wM$ is called stable if all formulas $\phi(x,y)$ are stable.

These notions also make sense, of course, for formulas defined on imaginary sorts, i.e., formulas of the structure $\wM^\meq$.

Stable independence is related to the possibility of having $\phi$-definitions that are all definable (modulo imaginaries and algebraicity) over some given smaller subset of $N$. More precisely, if $M$ and $N$ are submodels of $\wM$ and $B\subseteq N$ is any subset, we say that $M$ is \emph{stably independent from $N$ over $B$} if for every $x$-tuple $a$ from $M^\meq$ and every stable formula $\phi(x,y)$ of $\wM^\meq$, the $\phi$-definition of $\tp(a/N^\meq)$ is $\acl(B)$-definable in $\wM^\meq$. In that case we write $M\ind_B N$, or simply $M\ind N$ if $B=\emptyset$.

\begin{rem}
Since later we will be particularly interested in stable formulas on imaginary sorts, we have chosen our definition of stable independence so that it is clear that one has $M\ind_{B}N$ in $\wM$ if and only if $M^\meq\ind_{B}N^\meq$ in $\wM^\meq$. It is possible to give an equivalent definition that quantifies only over formulas $\phi(x,y)$ in the language of $\wM$: for this one has to consider all formulas $\phi(x,y)$ that are stable when restricted to a product $X\times Y$ of some definable subsets $X\subseteq\wM^x$, $Y\subseteq\wM^y$.
\end{rem}

We state below some of the main properties of the ternary relation $\ind$ in the general unstable setting. For comparison with the globally stable case see \cite[\textsection 8.2]{benusv10}.

\begin{prop}\label{prop:prop-of-ind}
The stable independence relation (within a large model of a possibly unstable metric theory) enjoys the following properties.

Below, $M$, $N$, etc.\ denote submodels, and $B$ is a subset of $N$. Moreover, the submodels and sets considered are given some implicit enumeration so that their types can be compared.
\begin{enumerate}
\item\label{item:invariance} (Invariance) If $M\ind_B N$ and $MBN\equiv M'B'N'$, then $M'\ind_{B'}N'$.
\item\label{item:monotonicity} (Monotonicity) If $M\subseteq M'$, $N\subseteq N'$ and $M'\ind_B N'$, then $M\ind_B N$.
\item\label{item:transitivity} (Transitivity) Suppose $N\subseteq N'$. Then $M\ind_B N'$ if and only if $M\ind_B N$ and $M\ind_N N'$.
\item\label{item:existence} (Existence) For every $M,B,N$ there exists $M'$ such that $M'\equiv_B M$ and $M'\ind_B N$.
\item\label{item:weak-sym} (Weak symmetry) Suppose $B\subseteq M$. If $M\ind_B N$, then $N\ind_B M$.
\item\label{item:anti-ref} (Stable anti-reflexivity) If $M\ind_B N$ and $S$ is a sort with a stable metric, then $M\cap N\cap S\subseteq\acl(B)$.
\end{enumerate}
\end{prop}
\begin{proof} Invariance and monotonicity are clear. For existence we can refer to Ben Yaacov's work \cite[Corollary~2.4]{benStabGroups}. The proofs of (\ref{item:transitivity}), (\ref{item:weak-sym}) and (\ref{item:anti-ref}) are very much like in the globally stable case, but we write them down here for the convenience of the reader ---see also Remark~\ref{rem:weak-sym} below for a discrepancy with the stable setting. We work within $\wM^\meq$ instead of $\wM$, so all submodels considered are elementary substructures of the former.

In (\ref{item:transitivity}), the left-to-right implication is clear. For the converse, let $a$ be an $x$-tuple from $M$ and take $p=\tp(a/N)\subseteq p'=\tp(a/N')$. If $M\ind_N N'$ and $\phi(x,y)$ is a stable formula, then $\dd{p}{\phi}$ and $\dd{p'}{\phi}$ are both $N$-definable predicates that agree on the submodel $N$; hence they coincide everywhere in $\wM^\meq$. Thus if $\dd{p}{\phi}$ is $\acl(B)$-definable then so is $\dd{p'}{\phi}$, i.e., if $M\ind_B N$ then $M\ind_B N'$ as well.

Now we prove (\ref{item:weak-sym}). Suppose $B\subseteq M\cap N$ and $M\ind_B N$. Let $c$ be an $x$-tuple from $N$, $q=\tp(c/M)$, and let $\phi(x,y)$ be a stable formula. Using existence, we can find $c'$ in $\wM^\meq$ such that $\tp(c/\acl(B))\subseteq q'=\tp(c'/M)$ and the $\phi$-definition of $q'$ is $\acl(B)$-definable. Now let $a\in M^y$, $p=\tp(a/M)$, and let us denote by $\psi(y,x)$ the formula $\psi(y,x)=\phi(x,y)$. Then we have:
$$\dd{q}{\phi}(a)=\phi(c,a) =\psi(a,c)=\dd{p}{\psi}(c)=\dd{p}{\psi}(c'),$$
where the last equality holds because $\dd{p}{\psi}$ is $\acl(B)$-definable and $c\equiv_{\acl(B)}c'$. On the other hand, by the continuous version of Harrington's lemma (see \cite[Proposition~7.16]{benusv10}), we have $\dd{p}{\psi}(c')=\dd{q'}{\phi}(a)$. Thus we get $\dd{q}{\phi}(a)=\dd{q'}{\phi}(a)$ for every $a$ in $M$, and we conclude that $\dd{q}{\phi}$ is $\acl(B)$-definable. Since this holds for every $c\in N^x$ and every stable $\phi$, we have $N\ind_B M$.

Finally, for (\ref{item:anti-ref}), suppose $M\ind_B N$ and $a\in M\cap N\cap S$. Let $\rho(x,y)$ be the metric in the sort $S$, which we assume to be stable. If $p=\tp(a/N)$, then $\dd{p}{\rho}(b)=\rho(a,b)$ for every $b\in N$, and $\dd{p}{\rho}$ is $\acl(B)$-definable. In other words, the function on $N\cap S$ that gives the distance to the element $a\in N\cap S$ is $\acl(B)$-definable, and this is equivalent to saying that $a\in\acl(B)$.
\end{proof}

\begin{rem}\label{rem:weak-sym}
In the property of weak symmetry, the assumption $B\subseteq M$ is essential and not just because our definition of $N\ind_B M$ required $B\subseteq M$. Indeed, one may give a natural definition of $A\ind_B C$ for arbitrary sets $A,B,C$, and with this definition the properties given above still hold (with more involved proofs). However, $A\ind_B C$ need not imply $C\ind_B A$ if $B$ is not contained in $A$, contrary to the situation in stable theories.
\end{rem}

\subsection{Independence in Hilbert spaces}

Hilbert spaces are a nice example of both $\aleph_0$-catego\-ricity and stability. There are many ways to present a (separable, infinite-dimensional, complex) Hilbert space~$H$ as a metric structure in the sense of continuous logic, all of them bi-interpretable. A minimal one is to take just the unit ball (or the unit sphere) of $H$ with no further structure than its metric and multiplication by the scalar~$i$. The inner product is definable from these, and all the other $n$-balls of $H$, as well as the vector space operations defined between them, become imaginary sorts and definable operations of this structure, respectively.

The theory of the resulting structure is $\aleph_0$-categorical: there is only one separable, infinite-dimensional, complex Hilbert space. It is also stable, and the associated independence relation coincides with the usual notion of orthogonality.

More precisely, given closed subspaces $H_0,H_1,H_2$ of a Hilbert space $H$ with $H_1\subseteq H_2$, let us write
$$H_0\bigperp_{H_1}H_2$$
to say that the orthogonal projection $\pi_2\colon H\to H_2$ agrees on $H_0$ with the orthogonal projection $\pi_1\colon H\to H_1$, that is, $\pi_2(a)=\pi_1(a)$ for every $a\in H_0$. When $H_1$ is the trivial subspace we write simply $H_0\bigperp H_2$, and this amounts of course to saying that the subspaces $H_0$ and $H_2$ are orthogonal.

\begin{prop}\label{prop:ind-vs-ort}
Let $H_0$ and $H_1\subseteq H_2$ be Hilbert subspaces of $H$. Then $H_0\bigperp_{H_1}H_2$ if and only if $H_0\ind_{H_1}H_2$.
\end{prop}
\begin{proof}
See, for instance, \cite[Theorem~15.8]{bbhu08}.
\end{proof}

Let us review some additional properties of the orthogonality relation. Given Hilbert subspaces $H_1,H_2$ of some ambient space $H$, we will denote by $H_1H_2\coloneqq\ov{\langle H_1,H_2\rangle}$ the closed subspace generated by them. Also, we will denote by $H_2\ominus H_1$ the orthogonal projection of $H_2$ to the orthogonal complement of $H_1$. Note that $H_2\ominus H_1\subseteq H_1H_2$.

We extend our definition of the orthogonality relation to include triples $H_0,H_1,H_2$ for which $H_1$ need not be a subspace of $H_2$: in that case $H_0\bigperp_{H_1}H_2$ means $H_0\bigperp_{H_1}H_1H_2$.

Now the following is an easy exercise.

\begin{lem}\label{lem:orthogonality-equivalence}
For any subspaces $H_0,H_1,H_2$ we have $H_0\bigperp_{H_1}H_2$ if and only if $H_0\ominus H_1\bigperp H_2\ominus H_1$.
\end{lem}

Finally, given vectors $a,b\in H$, we may write expressions like $a\bigperp_{H_1}H_2$ or $a\bigperp_{H_1}b$, which mean that the associated one-dimensional subspaces satisfy the corresponding orthogonality relations. The independence relation in Hilbert spaces enjoys the following very strong property.

\begin{prop} (Triviality) For any Hilbert subspaces $H_0,H_1,H_2$ we have $H_0\bigperp_{H_1}H_2$ if and only if $a\bigperp_{H_1}b$ for every pair of vectors $a\in H_0$ and $b\in H_2$.
\end{prop}
\begin{proof}
If the relation $H_0\bigperp_{H_1}H_2$ fails, then some $a\in H_0$ does not satisfy $a\bigperp_{H_1}\pi_{H_2}(a)$. The converse is clear.
\end{proof}

\noindent\hrulefill

\section{Intervals in the free group}\label{sInt}

Let $\mbF$ be the free group in the generators $\set{a,b}$, with identity $e$. We denote
$$\mbL\coloneqq\set{a,a\inv,b,b\inv}\subseteq\mbF.$$
It will be useful to define, for each $k\in\mbN$, $0\leq i<4$ and $n=4k+i$, an element $\ell_n\in \mbL$ by:
$$\ell_n=\begin{cases} a & \text{if } i=0, \\ a\inv & \text{if } i=1, \\ b & \text{if } i=2, \\ b\inv & \text{if } i=3.\end{cases}$$
On the other hand, let $[\mbF]^{<\omega}$ be the set of finite subsets of $\mbF$ together with the left action $\mbF\actson [\mbF]^{<\omega}$, $gS=\set{gs:s\in S}$.

\begin{defin}
We define inductively an increasing chain $\mbI_0=\set{I_n}_{n\in\mbN}\subseteq [\mbF]^{<\omega}$ of finite subsets of $\mbF$, as follows. We let $I_0=\set{e}$. Then, if $I_n$ has already been defined, we set
$$I_{n+1} = I_n\cup \ell_n I_n.$$
Let $\mbI\coloneqq \mbF \mbI_0\cup\set{\emptyset}$ be the set of all translates of sets of the chain, plus the empty set. We call the elements of $\mbI$ \emph{intervals} of $\mbF$.
\end{defin}

Thus, for instance, the first six sets of $\mbI_0$ are:
$${\small\left\lbrace e\right\rbrace
\subseteq \left\lbrace e,a\right\rbrace
\subseteq \left\lbrace a\inv,e,a \right\rbrace
\subseteq \left\lbrace\begin{matrix} a\inv,e,a\\ ba\inv,b,ba\end{matrix}\right\rbrace
\subseteq \left\lbrace \begin{matrix} b\inv a\inv, b\inv, b\inv a \\ a\inv, e, a\	\\ ba\inv, b, ba\	\end{matrix} \right\rbrace
\subseteq \left\lbrace \begin{matrix} b\inv a\inv, b\inv, b\inv a, ab\inv a\inv, ab\inv, ab\inv a\\ a\inv, e, a, a^2\	\\ ba\inv, b, ba, aba\inv, ab, aba\	\end{matrix} \right\rbrace.}$$
It is easy to see that $\mbI_0$ is strictly increasing and exhausts $\mbF$, i.e., $I_n\subsetneq I_{n+1}$ for every $n\in\mbN$ and $\bigcup_{n\in\mbN}I_n=\mbF$.

We let $\ov{\mbF}$ be the semigroup of words in the letters $\ov{a}$, $\ov{b}$, $\ov{a\inv}$, $\ov{b\inv}$, including the empty word. The concatenation of two words $\alpha,\beta\in \ov{\mbF}$ will be denoted by $\alpha*\beta$. Given an element $w\in\mbF$, we let $\ov{w}\in\ov{\mbF}$ be its unique associated \emph{reduced} word.

\begin{lem}\label{sInt:lem:w=u*v}
Suppose $\ov{w}=\ov{u}*\ov{v}$ and $w\in I_n$. Then $u\in I_n$ and $v\in I_n$.
\end{lem}
\begin{proof}
We proceed by induction on $n$, the case $n=0$ being obvious. Suppose the statement holds for $n$; let $w\in I_{n+1}$ satisfy $\ov{w}=\ov{u}*\ov{v}$ and let us show that $u,v\in I_{n+1}$. By the induction hypothesis, we may assume that $w\in I_{n+1}\setminus I_n\subseteq \ell_n I_n$, so that $w=\ell_n w_1$ for some $w_1\in I_n$. Moreover, $\ov{\ell_n}*\ov{w_1}$ is a reduced word, for otherwise $\ov{w_1}=\ov{\ell_n\inv}*\ov{w_2}$ for some $w_2$ which, by the inductive hypothesis, must belong to $I_n$; hence $w=w_2\in I_n$, which contradicts our assumption on $w$. Thus we have $\ov{w}=\ov{\ell_n}*\ov{w_1}$.

Now, we may assume $u\neq e$, hence $\ov{w}=\ov{\ell_n}*\ov{u_1}*\ov{v}$ for some $u_1$ such that $u=\ell_n u_1$. It follows that $\ov{w_1}=\ov{u_1}*\ov{v}$ and, by the inductive hypothesis, $u_1,v\in I_n$. Hence $u,v\in I_{n+1}$, as desired.
\end{proof}

We single out the following property that appeared in the previous proof. Let us say that an element $w\in\mbF$ \emph{begins with} $\ell\in \mbL$ if we have $\ov{w}=\ov{\ell}*\ov{u}$ for some $u\in\mbF$.

\begin{lem}\label{sInt:lem:first-letter}
Let $n\in\mbN$ and $\ell\in \mbL$. If $w\in \ell I_n\setminus I_n$, then $w$ begins with $\ell$.
\end{lem}
\begin{proof}
Let $w=\ell u\in \ell I_n$. If $\ov{w}\neq\ov{\ell}*\ov{u}$, then $\ov{u}=\ov{\ell\inv}*\ov{w}$ and, by the previous lemma, $w\in I_n$.
\end{proof}

\begin{prop}\label{sInt:prop:basic-intersections}
Let $n\in\mbN$.
\begin{itemize}[leftmargin=20pt]
\item If $n$ is even, then 
$I_n\cap \ell_nI_n = I_{n-3}$.
\item If $n$ is odd, then
$I_{n}\cap \ell_nI_n = I_{n-1}$.
\end{itemize}
(With the convention that $I_{-3}=I_{-1}=\emptyset$.)
\end{prop}
\begin{proof}
We start with the even case $n=2m$. For concreteness we assume $m$ is also even, so that $\ell_n=a$, but everything goes through identically when $m$ is odd, by exchanging the roles of $a$ and $b$ below. 

Consider:
\vspace{-8pt}
\begin{align*}
& A=I_{n-4},\ \ A_1=aA\setminus A,\ \ A_{-1}=a\inv A\setminus A, \\
& B=A_{-1}\cup A\cup A_1,\ \ B_1=bB\setminus B,\ \ B_{-1}=b\inv B\setminus B, \\
& C=A\cup A_1=a(A_{-1}\cup A).
\end{align*}
We note that $C=I_{n-3}$, $B=I_{n-2}$, and $B_{-1}\cup B\cup B_1=I_n$. We want to show that the intersection
$$I_n\cap \ell_n I_n = (B_{-1}\cup B \cup B_1)\cap (aB_{-1} \cup aB \cup aB_1)$$
is precisely $C$. By Lemma~\ref{sInt:lem:first-letter}, the elements of $B_{-1}$ begin with $b\inv$ and the elements of $B_1$ begin with $b$, which implies that $B_{-1}$ and $B_1$ are disjoint from $aB_{-1}$ and $aB_1$. On the other hand, we have $aB=C\cup aA_1$. Since $B_{-1}$ and $B_1$ are disjoint from $C\subseteq B$ and since the elements of $A_1$ begin with $a$, we deduce that $B_{-1}$ and $B_1$ are also disjoint from $aB$. Similarly, we have $B=A_{-1}\cup C$ and since $aB_{-1}$ and $aB_1$ are disjoint from $C\subseteq aB$, and since the elements of $A_{-1}$ begin with $a\inv$, we have that $aB_{-1}$ and $aB_1$ are disjoint from $B$. Thus the former intersection becomes:
$$I_n\cap \ell_n I_n = B\cap aB = (A_{-1}\cup C) \cap (C \cup aA_1).$$
As before, the elements of $A_{-1}$ begin with $a\inv$ and the elements of $A_1$ begin with $a$. Hence $A_{-1}$ and $aA_1$ are disjoint from each other and from $C$, and it follows that $I_n\cap \ell_n I_n = C$, as desired.

The odd case of the statement is simpler. For convenience, in order to reuse the notation of the previous case, we prove that $I_{n-3}\cap \ell_{n-3}I_{n-3}=I_{n-4}$ where $n=2m$ and $m$ is even (again, the proof when $m$ is odd is identical exchanging $a$ and $b$). In other words, we want to show that
$$C\cap a\inv C = A.$$
Now, the left-hand side is equal to $(A\cup A_1)\cap (A_{-1}\cup A)$, and the equality follows from the fact that $A_{-1}$ and $A_1$ are disjoint from each other and from $A$.
\end{proof}

If both $I$ and $J$ are in $\mbI$ and we have $J\subseteq I$, we say $J$ is a \emph{subinterval} of $I$ and write $J\leq I$.

\begin{lem}\label{sInt:lem:subintervals-of-I_n+1}
Every proper subinterval of $I_{n+1}$ is a subinterval of $I_n$ or of $\ell_n I_n$.
\end{lem}
\begin{proof}
Let $wI_m$ be a proper subinterval of $I_{n+1}$. Then, by cardinality, $m\leq n$. If $wI_m$ is not a subinterval of $I_n$, then $w\neq e$ and there is $v\in I_m$ such that $wv\in \ell_n I_n\setminus I_n$. We observe that we cannot have $\ov{v}=\ov{w\inv}*\ov{v'}$ for some $v'$, for otherwise we would have $wv=v'\in I_m\subseteq I_n$ by Lemma~\ref{sInt:lem:w=u*v}, contradicting our hypothesis on $wv$. Hence, using Lemma~\ref{sInt:lem:first-letter}, we deduce that $\ov{w}=\ov{\ell_n}*\ov{u}$ for some~$u$.

Now this implies that $wI_m\leq \ell_nI_n$. Indeed, if this is not the case then there is $t\in I_m$ such that $\ell_n ut\in I_n\setminus \ell_n I_n$ or, equivalently, $ut\in \ell_n\inv I_n\setminus I_n$. Using lemmas~\ref{sInt:lem:w=u*v} and \ref{sInt:lem:first-letter} as before, we deduce that $u$ begins with $\ell_n\inv$, contradicting the fact that $\ov{w}=\ov{\ell_n}*\ov{u}$.
\end{proof}

\begin{prop}
The set of intervals is closed under intersections.
\end{prop}
\begin{proof}
Since intervals are finite and the increasing chain $\mbI_0$ exhausts $\mbF$, it is enough to check that for every $n\in\mbN$ the set of subintervals of $I_n$ is closed under intersections. We prove this by induction on $n$.

The base case is trivial. Suppose the claim holds up to $n$, and let $I$ and $I'$ be proper subintervals of $I_{n+1}$. Then, by the previous lemma, $I\leq I_n$ or $I\leq \ell_n I_n$, and similarly for $I'$. If both $I$ and $I'$ subintervals of $I_n$, then we are done by the inductive hypothesis. Similarly, if they are both subintervals of $\ell_n I_n$, then $\ell_n\inv I$ and $\ell_n\inv I'$ are subintervals of $I_n$; hence $\ell_n\inv I \cap \ell_n\inv I'$ is an interval by the inductive hypothesis, and so is $I\cap I'$ being a translate of it.

Thus we may assume that $I\leq I_n$ and $I'\leq \ell_n I_n$. By Proposition~\ref{sInt:prop:basic-intersections}, we then have
$$I\cap I' = I\cap I' \cap I_n\cap\ell_n I_n = I\cap I' \cap I_k,$$
where $k=n-3$ if $n$ is even and $k=n-1$ if $n$ is odd. Now, by the inductive hypothesis, the set $J=I\cap I_k$ is a subinterval of $I_n$. Similarly, $\ell_n\inv I'\cap \ell_n\inv I_k\leq I_n$, and so $J'=I'\cap I_k$ is a subinterval of $\ell_n I_n$. In particular, both $J$ and $J'$ are subintervals of $I_k$ and thus, by the inductive hypothesis again, their intersection $J\cap J'=I\cap I'$ is a subinterval as well.
\end{proof}

\noindent\hrulefill

\section{Cairns of models}\label{sCai}

In this section we will pile up several copies of a given $\aleph_0$-categorical structure to form a larger copy endowed with a convenient $\mbF$-action.

Throughout the section we fix an $\aleph_0$-categorical metric structure $M$ and a sufficiently saturated and homogeneous elementary extension $\wM\succeq M$. Then, by \emph{a copy of $M$} we mean any other elementary embedding of $M$ into $\wM$ or, which is the same, any tuple $M'=(a_m)_{m\in M}\subseteq\wM$ having the same type over $\emptyset$ as $(m)_{m\in M}$. (If the reader prefers, we might as well consider only countable tuples, by replacing the index set by some countable dense subset of~$M$.) If $M_0$ is a copy of~$M$, we will by abuse of notation denote its image, which is an elementary substructure of $\wM$, again by $M_0$. Observe that if $M_0$ and $M_1$ are copies of $M$, then $M_0\subseteq M_1$ if and only if $M_0\preceq M_1$.

We recall that $\mbI$ denotes the set of intervals of the two-generated free group $\mbF$ as defined in the previous section.

\begin{prop}\label{sCai:prop:cairns}
There exists a family $\set{M_I}_{I\in\mbI}$ of tuples in $\widehat{M}$ satisfying the following properties:
\begin{enumerate}
\item If $I\neq\emptyset$, $M_I$ is a copy of $M$.
\item $M_\emptyset=\emptyset$.
\item If $J\leq I$, $J\neq\emptyset$, then $M_J\preceq M_I$.
\item $(M_I)_{I\in\mbI}\equiv (M_{aI})_{I\in\mbI}\equiv (M_{bI})_{I\in\mbI}$.
\item If $I,J\in\mbI$ and $K=I\cap J$, then $M_I\ind_{M_K} M_J$.
\end{enumerate}
\end{prop}
\begin{proof}
We set $M_\emptyset\coloneqq \emptyset$ and let $M_{I_0}$ be any copy of $M$. Suppose inductively that for some $n\in\mbN$ we have produced tuples $M_I$ for each $I\leq I_n$ satisfying the following conditions:
\begin{enumerate}[label=(\roman*)]
\item\label{item1:prop:cairn} If $J\leq I\leq I_n$, $J\neq\emptyset$, then $M_J$ is a copy of $M$ and $M_J\preceq M_I$.
\item\label{item2:prop:cairn} For every $m<n$, $(M_I)_{I\leq I_m}\equiv (M_{\ell_m I})_{I\leq I_m}$.
\item\label{item3:prop:cairn} If $I,J\leq I_n$ and $K=I\cap J$, then $M_I\ind_{M_K} M_J$.
\end{enumerate}
We want to construct $M_J$ for every $J\leq I_{n+1}$ that is not already a subinterval of $I_n$, in such a way that the preceding conditions also hold for $n+1$. Let $\ell=\ell_n$. Since, by Lemma~\ref{sInt:lem:subintervals-of-I_n+1}, every $J\leq I_{n+1}$ is either a subinterval of $I_n$, a subinterval of $\ell I_n$, or $I_{n+1}$ itself, we have to construct $M_J$ for $J=I_{n+1}$ and for those $J$ of the form $J=\ell I$, $I\leq I_n$, not already contained in $I_n$.

We let $k=n-3$ if $n$ is even and $k=n-1$ if $n$ is odd; in either case we have $\ell_k=\ell\inv$ and $I_n\cap \ell I_n = I_k$ by Proposition~\ref{sInt:prop:basic-intersections}. Let $C=(M_I)_{I\leq I_k}$ and $C'=(M_{\ell\inv I})_{I\leq I_k}$. We know inductively that $C\equiv C'$, so let $\tau$ be an automorphism of $\wM$ sending $C'$ to $C$. Let moreover $D=(M_I)_{I\leq I_n}$. Then, using the existence property of independence, choose $D'=(N_I)_{I\leq I_n}$ in $\wM$ such that
\begin{equation}\label{eq:prop:cairn}\tag{*}
D'\equiv_{M_{I_k}} \tau D\text{ \ and \ }N_{I_n}\ind_{M_{I_k}} M_{I_n}.
\end{equation}
Now, for every $I\leq I_n$ such that $\ell I$ is not contained in $I_n$, we set $M_{\ell I}\coloneqq N_I$. On the other hand, if $I\leq I_n$ is such that $\ell I\leq I_n$, then we have
$$\ell I\leq I_n\cap \ell I_n = I_k.$$
In particular, $M_{\ell I}\subseteq M_{I_k}$ and $\tau$ sends $M_I=M_{\ell\inv \ell I}$ to $M_{\ell I}$. But then, by the first condition in~(\ref{eq:prop:cairn}), $M_{\ell I} = \tau M_I \equiv_{M_{I_k}} N_I$, and therefore $M_{\ell I} = N_I$. In other words, we have $M_{\ell I} = N_I$ for every $I\leq I_n$. Since $D'\equiv\tau D\equiv D$, we obtain that $(M_I)_{I\leq I_n}\equiv (M_{\ell I})_{I\leq I_n}$, which is the case of condition \ref{item2:prop:cairn} that we had to ensure in the inductive step.

Now we define $M_{I_{n+1}}$ to be any copy of $M$ that contains both $M_{I_n}$ and $M_{\ell I_n}$. This exists because $M$ is $\aleph_0$-categorical. Then condition~\ref{item1:prop:cairn} is satisfied by construction. We are left to check condition~\ref{item3:prop:cairn}.

Let $I,J\leq I_{n+1}$, $K=I\cap J$. If either of $I,J$ is equal to $I_{n+1}$, there is nothing to prove. If both $I,J$ are subintervals of $I_n$ or of $\ell I_n$, then the condition follows from the inductive hypothesis and, in the latter case, from the fact that $(M_I)_{I\leq I_n}\equiv (M_{\ell I})_{I\leq I_n}$ and the invariance of independence. Finally, if $I\leq I_n$ and $J\leq \ell I_n$ then $K\leq I_k$ and, by construction and monotonicity,
$$M_J\ind_{M_{I_k}} M_{I_n}.$$
On the other hand, if we let $Q=J\cap I_k\leq \ell I_n$, then by the inductive hypothesis and invariance we have $M_J\ind_{M_Q} M_{I_k}$ and so, by transitivity and monotonicity,
$$M_J\ind_{M_Q} M_I.$$
Since $M_Q\subseteq M_J$, by symmetry we have $M_I\ind_{M_Q} M_J$. Finally, since both $Q$ and $I$ are subintervals of $I_n$ and since $Q\cap I=K$, we also have inductively that $M_I\ind_{M_K}M_Q$. Hence
$M_I\ind_{M_K}M_J$ by transitivity, and also $M_J\ind_{M_K}M_I$ by symmetry, as desired.

Thus we have built a family of tuples $(M_I)_{I\in\mbI}$ satisfying the conditions \ref{item1:prop:cairn}, \ref{item2:prop:cairn}, \ref{item3:prop:cairn} above. On the other hand, since every finite set of intervals is contained in the set of subintervals of $I_m$ for some $m\in\mbN$ with $\ell_m=a$, and since from \ref{item2:prop:cairn} we have $(M_I)_{I\leq I_m}\equiv (M_{aI})_{I\leq I_m}$, we deduce that $(M_I)_{I\in\mbI}\equiv (M_{aI})_{I\in\mbI}$. Similarly, $(M_I)_{I\in\mbI}\equiv (M_{bI})_{I\in\mbI}$. We conclude that $(M_I)_{I\in\mbI}$ has all the properties of the statement.
\end{proof}

Suppose now that we have a family $\set{M_I}_{I\in\mbI}$ as given by the proposition, which we may see as a direct system of elementary embeddings. Then the closed union
$$M_\infty=\ov{\bigcup_{I\in\mbI}M_I},$$
being equal to the closed union of the elementary chain $\set{M_I}_{I\in\mbI_0}$, is a common elementary extension of every $M_I$. Moreover, by $\aleph_0$-categoricity, $M_\infty\simeq M$.

On the other hand, the conditions $(M_I)_{I\in\mbI}\equiv (M_{aI})_{I\in\mbI}$ and $(M_I)_{I\in\mbI}\equiv (M_{bI})_{I\in\mbI}$ imply that there are automorphisms $\tau_a$, $\tau_b$ of $\widehat{M}$ such that $\tau_a M_I = M_{aI}$ and $\tau_b M_I = M_{bI}$ for every $I\in\mbI$. Note that, then, $\tau_a\inv M_I = M_{a\inv I}$ and $\tau_b\inv M_I = M_{b\inv I}$. In particular, these automorphisms restrict to automorphisms of $M_\infty$.

Thus we can phrase the previous proposition in the following manner.

\begin{theorem}\label{sCai:thm:cairns}
Every $\aleph_0$-categorical structure $M$ admits an action $\mbF\actson^\tau M$ by automorphisms and a direct system of elementary substructures $C=\set{M_I}_{I\in\mbI}$ (with base $M_\emptyset=\emptyset$) such that:
\begin{enumerate}
\item $M$ is the direct limit of $C$.
\item For every $w\in\mbF$ and $I\in\mbI$, $\tau_w M_I = M_{wI}$.
\item If $I,J\in\mbI$ and $K=I\cap J$, then $M_I\ind_{M_K} M_J$.
\end{enumerate}
\end{theorem}

\begin{rem}\label{rem:anti-ref-free-action}
Suppose $S$ is a sort of $M$ with a stable metric. Then, by the anti-reflexivity of independence, the $\mbF$-action given in Theorem~\ref{sCai:thm:cairns} is free on $\bigcup_{I\in\mbI} S_I\setminus\acl(\emptyset)$, where $S_I$ denotes the corresponding sort of $M_I$. Indeed, if $a\in S_I$ is such that $\tau_w a=a$ for some $w\neq e$, then there is $n\in\mbN$ with $I\cap w^nI=\emptyset$. Hence $M_I\ind M_{w^nI}$ and $a\in S_I\cap S_{w^nI}=\acl(\emptyset)$.

In particular, if $M$ is a \emph{classical} structure, then the action $\mbF\actson^\tau M$ is free on $M\setminus\acl(\emptyset)$.
\end{rem}

A similar argument as in the previous remark yields the following.

\begin{cor}\label{cor:WAP-nonAP}
Let $G$ be a Roelcke precompact Polish group, say the automorphism group of an $\aleph_0$-categorical structure $M$. Suppose the Bohr compactification of $G$ is a proper factor of its WAP compactification. Then the action $\tau$ of Theorem~\ref{sCai:thm:cairns} induces a discrete copy of a rank two non-abelian free group inside $G$.
\end{cor}
\begin{proof}
Let $\widetilde G$ be the homomorphic image of $G$ inside its WAP compactification. Then the topology of $\widetilde G$ is induced by weakly almost periodic functions and thus, by the main theorem of Ben Yaacov--Berenstein--Ferri \cite{benReflexive}, $\widetilde G$ admits a \emph{stable} left-invariant metric $\rho$. It follows that the left-completion $S$ of $\widetilde G$ with respect to $\rho$ is an imaginary sort of $M$ with a stable metric. Note that $S$ is not compact, for otherwise $\widetilde G$ would be compact and equal to both the Bohr and the WAP compactifications of $G$. The elements of $S$ are thus not algebraic because their closed $G$-orbit is not compact.

By Remark~\ref{rem:anti-ref-free-action} (and using its notation), the action $\mbF\actson^\tau \bigcup_{I\in\mbI}S_I$ is free, so $\tau(\mbF)$ is a rank two non-abelian free subgroup of $G$. To see that $\tau(\mbF)$ is discrete in $G$, suppose to the contrary that we have a sequence $w_k\in\mbF\setminus\set{e}$ with $\tau_{w_k}\to 1_G$. Choose any $I\in\mbI$ and $a\in S_I$. By freeness of the action, the size of the $w_k$ (as words) tends to infinity, so there is $k_0$ such that for all $k\geq k_0$ we have $w_kI\cap I=\emptyset$. Thus also $M_{w_kI}\ind M_I$. This implies that for $k\geq k_0$ the $\rho$-definitions of the types $\tp(\tau_{w_k}a/M_I)$ are $\acl(\emptyset)$-definable. Since $\tau_{w_k}a\to a\in M_I$, it follows easily that $a\in\acl(\emptyset)$, contradicting the non-algebraicity of the elements of $S$.
\end{proof}

We end this section with a few side comments.

\begin{rem}
The construction of this section does not use much of the $\aleph_0$-categoricity hypothesis. Indeed, everything goes through, for instance, if $M$ is just assumed to be separable and approximately $\aleph_0$-saturated.

Similarly, not much is used in the construction about the stable independence relation. The latter could be replaced by any ternary relation $\ind^*$ satisfying the properties of invariance, monotonicity, transitivity, existence and weak symmetry (see Proposition~\ref{prop:prop-of-ind}).

On the other hand, $\aleph_0$-categoricity and stability will both be crucial in the next section.
\end{rem}

\begin{rem}
One could present the construction from the topological group viewpoint. Indeed, if $M$ is $\aleph_0$-categorical then the semigroup $\End(M)$ of elementary self-embedding of $M$ can be identified with the left-completion $\widehat{G}_L$ of its automorphism group; hence we can see our system $\set{M_I}_{I\in\mbI}$ as a subset of $\widehat{G}_L$ satisfying certain algebraic relations. The independence conditions $M_I\ind_{M_K} M_J$, for instance, translate to equations involving the $*$-semigroup laws of the WAP compactification $W(G)\supseteq\widehat{G}_L$, as per \cite[\textsection 5]{bentsa}. However, the model-theoretic approach is more natural to~us for the purposes of this paper.
\end{rem}

\begin{rem}
In connection with Corollary~\ref{cor:WAP-nonAP}, Todor Tsankov made me observe that, in the classical setting, one can use the method of Ehrenfeucht--Mostowski models (i.e., models built upon an indiscernible sequence) to see that every infinite oligomorphic group contains a closed copy of $\Aut(\mbQ,<)$ (which, in turn, contains a discrete non-abelian free group). The standard construction of Ehrenfeucht--Mostowski models in classical logic relies on Skolem functions, a technique that does not work in continuous logic. However, as pointed out subsequently to us by Ita{\"\i} Ben Yaacov, one can still produce Ehrenfeucht--Mostowski models for metric theories, and in particular show that every non-compact PRP group contains a closed copy of $\Aut(\mbQ,<)$. His construction has appeared in \cite[\textsection 3.7]{benCatcat}, in the framework of compact abstract theories. Since the proof can be presented in a way that is close in spirit to the constructions considered in this paper, we sketch it out below.
\end{rem}

\begin{theorem}\label{thm:AutQ-in-PRP}
The group $\Aut(\mbQ,<)$ embeds in every non-compact PRP group.
\end{theorem}
\begin{proof}
Let $\mbP$ be the set of finite subsets of $\mbQ$ (or $\mbN$, or any other countable linear order). The standard Ramsey-theoretic argument to show the existence of indiscernible sequences $(a_i)_{i\in\mbQ}$ (realizing a given partial Ehrenfeucht--Mostowski type) also shows the existence of systems $(a_F)_{F\in\mbP}$ (realizing a given partial \guillemotleft Ehren\-feucht--Mos\-tows\-ki $\mbP$-type\guillemotright, defined in the natural way) that are indiscernible in the following sense: for every strictly increasing function $\alpha\colon\mbQ\to\mbQ$, which we extend to $\alpha\colon\mbP\to\mbP$ by $\alpha F=\set{\alpha(i):i\in F}$, we have
$$(a_F)_{F\in\mbP}\equiv (a_{\alpha F})_{F\in\mbP}.$$

Let $M$ be $\aleph_0$-categorical, non-compact. Restricting the index set to $\mbP_n = \set{F\in\mbP:|F|\leq n}$ and proceeding inductively on $n$, we can build an indiscernible system $(M_F)_{F\in\mbP}$ of \emph{distinct} copies of $M$ such that $M_F\subseteq M_{F'}$ whenever $F\subseteq F'$. If we identify $M$ with the closed union of the system, then each strictly increasing $\alpha\colon\mbQ\to\mbQ$ induces an elementary embedding $\tau_\alpha\colon M\to M$ uniquely defined by: $\tau_\alpha M_F=M_{\alpha F}$ (as tuples) for every $F\in\mbP$. In particular we get an action $\Aut(\mbQ,<)\actson^\tau M$, which is moreover continuous and faithful for the pointwise convergence topology of $\Aut(\mbQ,<)$.
\end{proof}

\noindent\hrulefill

\section{Property (T) for automorphism groups of $\aleph_0$-categorical metric structures}\label{sec:PropT}

\subsection{From independence to orthogonality}\label{subsec:Indep-to-Ortho}
Before discussing Property (T) we prove a key proposition that is also interesting in itself.

A \emph{unitary representation} of a topological group $G$ is a continuous action of $G$ on a complex Hilbert space by unitary (i.e., linear isometric) transformations.

Suppose $G$ is the automorphism group of an $\aleph_0$-categorical structure $M$, and suppose we are given a unitary representation $G\actson^\sigma H$ and a distinguished unit vector $\xi\in H$. Then $O=\ov{\sigma(G)\xi}$, the orbit closure of~$\xi$, is an imaginary sort of $M$ as per Lemma~\ref{lem:G-actson-O}. Thus, given an elementary self-embedding $e\colon M\to M$, we can consider $e(O)\subseteq O$, the image of $O$ under the canonical extension of $e$ to $M^\meq$.

Now let $M_I$, $M_J$, $M_K$ be elementary substructures of $M$ (that is, the images of some given elementary self-embeddings of $M$) and let $O_I,O_J,O_K\subseteq O$ be the corresponding imaginary sorts as above. We let
$$H_I=\ov{\gen{O_I}}\subseteq H$$
be the closed span of $O_I$, and we define $H_J$ and $H_K$ similarly.

\begin{prop}\label{prop:indep-transfer} In the situation described above, the following hold.
\begin{enumerate}
\item Suppose $M_K\preceq M_J$. If $M_I\ind_{M_K}M_J$, then $H_I\bigperp_{H_K}H_J$.
\item Suppose $M$ satisfies $\dcl(\emptyset)=\acl(\emptyset)$ and the representation $G\actson^\sigma H$ has no invariant unit vectors. If $M_I\ind M_J$, then $H_I\bigperp H_J$.
\end{enumerate}
\end{prop}

\noindent Note that although this is reminiscent of Proposition~\ref{prop:ind-vs-ort}, it cannot be deduced directly from it because the full Hilbert space (i.e., its full unit sphere) need not be interpretable in $M$. 

\begin{proof}
(1) We suppose $M_K\preceq M_J$ and $M_I\ind_{M_K}M_J$. Take an element $a\in O_I$ and consider its orthogonal projection $\pi_J(a)$ to $H_J$. We want to show that $\pi_J(a)\in H_K$. Suppose for a contradiction that the orthogonal projection $\zeta$ of $\pi_J(a)$ to the orthogonal complement of $H_K$ has norm $\norm{\zeta}>\epsilon^{1/2}$ for some $\epsilon>0$. In particular, $\ip{a,\zeta}=\norm{\zeta}^2>\epsilon$.
Choose tuples $b\in O_J^n$, $c\in O_J^m$ and scalars $\lambda_i,\mu_j$ yielding approximations:
$$\textstyle\norm{\pi_J(a)-\sum\lambda_i b_i}<\delta\text{ \ and \ }\norm{\zeta-\sum\mu_j c_j}<\epsilon/3,$$
where $\delta>0$ is small enough that $\delta\sum\abs{\mu_j}<\epsilon/3$. We may as well suppose $\norm{\sum\lambda_i b_i}<1$. Note that we have $\abs{\ip{a,\sum\mu_jc_j}}>2\epsilon/3$ and $\abs{\ip{\chi,\sum\mu_j c_j}}<\epsilon/3$ for every $\chi\in H_K$ with $\norm{\chi}\leq 1$.

Now let $\phi(x,y)$ be the predicate in the variables $x$, $y=(y_j)_{j<m}$ of the sort $O$ given by:
$$\textstyle\phi(x,y)=\ip{x,\sum\mu_jy_j}.$$
Since $\phi\colon O^{1+m}\to\mbC$ is continuous and $G$-invariant, this is indeed a formula of $M^\meq$. Moreover, by the stability of the Hilbert space, $\phi(x,y)$ is stable. Let $p=\tp(a/M_J)$. Since
$$M_I\ind_{M_K}M_J,$$
the definition $\dd{p}{\phi}(y)$ is $M_K$-definable. Note that for $d\in O_J^m$ we have $\dd{p}{\phi}(d)=\ip{\pi_J(a),\sum\mu_jd_j}$. On the other hand, let us also consider the formula given by $\psi(z,y)=\ip{\sum\lambda_i z_i,\sum\mu_j y_j}$. Then for every $d\in O_J^m$,
$$\textstyle \bigabs{\dd{p}{\phi}(d)-\psi(b,d)} = \bigabs{\ip{\pi_J(a)-\sum\lambda_i b_i,\sum\mu_j d_j}} \leq \norm{\pi_J(a)-\sum\lambda_i b_i}\cdot\norm{\sum\mu_j d_j}\leq\delta\sum\abs{\mu_j}.$$
In other words,
$$\textstyle M_J^\meq\models \sup_y \bigabs{\dd{p}{\phi}(y)-\psi(b,y)}\leq \delta\sum\abs{\mu_j}.$$
Since $\dd{p}{\phi}(y)$ is $M_K$-definable and $M_K\preceq M_J$, we deduce that there is $b'\in O_K^m$ such that $\norm{\sum\lambda_i b_i'}\leq 1$ and
$$\textstyle\bigabs{\dd{p}{\phi}(d)-\psi(b',d)} <\epsilon/3$$
for every $d\in O_K^m$, and in fact for every $d\in O^m$. Specializing in $d=c$ we obtain
$$\textstyle\bigabs{\ip{a,\sum\mu_jc_j}-\ip{\sum\lambda_i b'_i,\sum\mu_j c_j}}<\epsilon/3,$$
which contradicts the fact that $\ip{a,\sum\mu_jc_j}>2\epsilon/3$ and $\ip{\chi,\sum\mu_j c_j}<\epsilon/3$ for every $\chi\in H_K$ with $\norm{\chi}\leq 1$. Thus we have shown that $a\bigperp_{H_K}H_J$ for every $a\in O_I$, and this implies $H_I\bigperp_{H_K}H_J$, as desired.

(2) Now we suppose every algebraic element is definable, $\sigma$ has no invariant unit vectors, and $M_I\ind M_J$. Let $a\in O_I$ and let us show that the orthogonal projection $\pi_J(a)\in H_J$ is $\sigma(G)$-invariant, so that by our hypothesis on the representation we get $\pi_J(a)=0$. Suppose to the contrary that there are $\epsilon>0$, $g\in G$ and $c\in O$ such that
$$\abs{\ip{\pi_J(a)-\sigma_g\pi_J(a),c}}>\epsilon.$$

We consider the formula given by the inner product, $\phi(x,y)=\ip{x,y}$, which is stable. Let as before $p=\tp(a/M_J)$. Since $M_I\ind M_J$, the definition $\dd{p}{\phi}(y)$ is $\acl(\emptyset)$-definable. But $M$ has $\acl(\emptyset)=\dcl(\emptyset)$, so it follows that $\dd{p}{\phi}(y)$ is $\emptyset$-definable and hence $G$-invariant.

Let $b\in O_J^n$ and $\lambda\in\mbC^n$ be such that $r=\norm{\pi_J(a)-\sum\lambda_i b_i}<\epsilon/4$. Define $\psi(z,y)=\ip{\sum\lambda_i z_i,y}$. Since $\dd{p}{\phi}(d)=\ip{\pi_J(a),d}$ for $d\in O_J$, we have, similarly as before:
$$\textstyle M_J^\meq\models \sup_y \abs{\dd{p}{\phi}(y)-\psi(b,y)}\leq r,$$
and therefore $\abs{\dd{p}{\phi}(d)-\psi(b,d)}<\epsilon/4$ for every $d\in O$ because $M_J\preceq M$. Now by the $G$-invariance of $\dd{p}{\phi}(y)$ we obtain that $\abs{\psi(b,d)-\psi(b,g\inv d)}<\epsilon/2$ for every $d\in O$. Hence also
$$\abs{\ip{\pi_J(a)-\sigma_g\pi_J(a),d}}=\abs{\ip{\pi_J(a),d}-\ip{\pi_J(a),g\inv d}} < \epsilon/4+\abs{\psi(b,d)-\psi(b,g\inv d)}+\epsilon/4 <\epsilon.$$
Specializing in $d=c$ we get a contradiction. We conclude that $\pi_J(a)=0$ for all $a\in O_I$ and thus $H_I\bigperp H_J$.
\end{proof}

\subsection{Property (T) and splitting Kazhdan sets}\label{subsec:Prop-(T)} A topological group $G$ has \emph{Property~(T)} if there are a compact subset $Q\subseteq G$ and some $\epsilon>0$ such that every unitary representation $G\actson^\sigma H$ \emph{with no invariant unit vectors} satisfies the stronger condition:
$$\max_{g\in Q}\norm{\sigma_g\xi-\xi}\geq\epsilon$$
for every unit vector $\xi\in H$. In that case, $Q$ is a \emph{Kazhdan set} and $(Q,\epsilon)$ is a \emph{Kazhdan pair} for $G$. If $Q$ can be chosen finite, then $G$ is said to have \emph{strong Property~(T)}.

As in \cite{evatsaFree}, we will use that Property~(T) is preserved under extensions of Polish groups.

\begin{lem}\label{sT:lem:extensions}
Let $G$ be a Polish group and $N$ be a closed normal subgroup of $G$. If $N$ and $G/N$ have Property~(T), then so does $G$. Moreover, if $N$ and $G/N$ have finite Kazhdan sets, then $G$ has a finite Kazhdan set as well.
\end{lem}
\begin{proof}
See Proposition 1.7.6 and Remark 1.7.9 of the book \cite{bekkaKazhdan}.
\end{proof}

Now suppose $G$ is the automorphism group of an $\aleph_0$-categorical structure. By virtue of Corollary~\ref{cor:Bohr}, $G$ has a smallest cocompact normal subgroup, $G^\circ$. Moreover, by Proposition~\ref{prop:Mcirc-is-cat}, $G^\circ$ is also the automorphism group of an $\aleph_0$-categorical structure, and in the latter every algebraic element is definable. Since compact groups have Property (T), Lemma~\ref{sT:lem:extensions} says that in order to prove $G$ has Property (T) it is enough to prove $G^\circ$ has Property (T). In other words, we may restrict our attention to automorphism groups of separably categorical structures for which $\dcl(\emptyset)=\acl(\emptyset)$.

Thus, for the rest of the section we fix an $\aleph_0$-categorical structure $M$ with $\dcl(\emptyset)=\acl(\emptyset)$ and let $G=\Aut(M)$. Moreover, we fix an action by automorphisms $\mbF\actson^\tau M$ and a \guillemotleft cairn\guillemotright\ of elementary self-embeddings $\set{M_I}_{I\in\mbI}\subseteq \End(M)\cup\set{\emptyset}$ as given by Theorem~\ref{sCai:thm:cairns}. We will show that $\set{\tau_a,\tau_b}$ is a Kazhdan set for $G$.

For the following two lemmas we fix in addition a unitary representation $G\actson^\sigma H$ with no invariant unit vectors. We assume moreover that $\sigma$ is \emph{cyclic}, i.e., that there is a unit vector $\xi\in H$ such that $H$ is the closed linear span of the orbit of $\xi$.

We apply the construction of \textsection \ref{subsec:Indep-to-Ortho} to the representation $\sigma$, the unit vector $\xi$, and the self-embeddings $M_I$. More precisely, let $O=\ov{\sigma(G)\xi}$ be the closed orbit of~$\xi$, which we see as an imaginary sort of $M$. For each non-empty $I\in\mbI$, we define $O_I$ as the image of $O$ by the elementary self-embedding $M_I\colon M^\meq\to M^\meq$. Thus $O_I\subseteq O\subseteq H$. We then let $H_I=\ov{\gen{O_I}}\subseteq H$ be the closed span of $O_I$. Finally, we let $H_\emptyset=\set{0}$ be the trivial subspace of $H$.

The independence conditions of the $M_I$ together with Proposition~\ref{prop:indep-transfer} yield the following.

\begin{lem}\label{sT:lem:ind-transfer}
For every $I,J\in\mbI$ and $K=I\cap J$, we have $H_I\bigperp_{H_K}H_J$.
\end{lem}

Now for every $n\in\mbN$ we consider
$$E_n=\ov{\gen{H_{wI_n}:w\in\mbF}},$$
the closed subspace generated by those $H_I$ such that $I$ is a translate of $I_n\in\mbI_0$. We define in addition $E_{-1}=\set{0}$, the trivial subspace. Note that since $M=\ov{\bigcup_{I\in\mbI} M_I}$, we also have $O=\ov{\bigcup_{I\in\mbI} O_I}$ and $H=\ov{\bigcup_{I\in\mbI} H_I}=\ov{\bigcup_{n\in\mbN} E_n}$.

On the other hand, since $\tau_w M_I = M_{wI}$ for every $w\in\mbF$ and $I\in\mbI$, for the induced action $\mbF\actson^\tau H$ we have
$$\tau_wO_I = O_{wI}\text{ \ and \ }\tau_wH_I = H_{wI}.$$
In particular, the subspaces $E_n$ are invariant under the action of $\mbF$.

\begin{lem}\label{lem:Hu-ort-En-Hv}
For every $n\in\mbN\cup\set{-1}$ and $u,v\in\mbF$ with $u\neq v$, we have $H_{uI_{n+1}}\bigperp_{E_n}H_{vI_{n+1}}$.
\end{lem}
\begin{proof}
For $n=-1$ this is a particular case of the previous lemma. Let $n\in\mbN$, $I=I_{n+1}$ and $\ell=\ell_n$. By invariance, it is enough to show that $H_I\bigperp_{E_n}H_{uI}$ for every $u\neq e$.

Let $H_I'=H_{I_n}H_{\ell I_n}$ be the subspace generated by $H_{I_n}$ and $H_{\ell I_n}$. Now, if $J$ is any interval not containing $I$ and we take $K=I\cap J$, then $K\leq I_n$ or $K\leq \ell I_n$. Hence $H_K\subseteq H_I'\subseteq H_I$, and since by Lemma~\ref{sT:lem:ind-transfer} we have $H_I\bigperp_{H_K}H_J$, we deduce that
$$H_I\bigperp_{H'_I}H_J.$$
This is true in particular for every $J=uI$ with $u\neq e$ as well as for every $J=wI_n$. By the triviality of the orthogonality relation, we obtain that $$H_I\bigperp_{H_I'}E_nH_{uI}$$ for every $u\neq e$. Since $H_I'\subseteq E_n$, we have $H_I\bigperp_{E_n}H_{uI}$, as desired.
\end{proof}

Let $\eta=\sqrt{2-\sqrt{3}}$ and let $\mbF\actson^\lambda \ell^2(\mbF)$ be the left-regular unitary representation of $\mbF$. Then one has, for every unit vector $\xi\in \ell^2(\mbF)$,
$$\max_{\ell\in\set{a,b}}\norm{\lambda_\ell\xi-\xi}\geq\eta.$$
Moreover, the same holds for every multiple of $\lambda$. That is, for every representation $\mbF\actson^\pi H$ where $\pi=\bigoplus\lambda$ is a (possibly infinite) direct sum of copies of $\lambda$ one has $\max_{\ell\in\set{a,b}}\norm{\pi_\ell\xi-\xi}\geq\eta$ for every unit vector $\xi\in H$. This follows from Kesten's computation \cite{kestenSym} of the operator norm of $\lambda_a+\lambda_{a\inv}+\lambda_b+\lambda_{b\inv}$, as shown in \cite[pp.~515-516]{bekkaUnitary}.

\begin{prop}
Let $Q=\set{\tau_a,\tau_b}$. Then $(Q,\eta)$ is a Kazhdan pair for $G$. Moreover, if $G\actson^\sigma H$ is a unitary representation of $G$ without invariant unit vectors, then $\sigma$ restricts to a multiple of the left-regular representation of the subgroup generated by~$Q$.
\end{prop}
\begin{proof}
We can split $\sigma$ into a direct sum of cyclic subrepresentations:
$$H=\bigoplus H_\xi.$$
Hence, by the previous discussion, it suffices to show that the induced representation $\mbF\actson^\pi H_\xi$ on each cyclic subspace $H_\xi$ is a multiple of the left-regular representation of $\mbF$.

Thus we assume that $G\actson^\sigma H=H_\xi$ is a cyclic representation with no invariant unit vectors, and define as before the subspaces $H_I$ and $E_n$ for every $I\in\mbI$ and $n\in\mbN\cup\set{-1}$. We define, furthermore, for every interval $J=wI_n$,
$$\widetilde H_J = H_J \ominus E_{n-1},$$
the orthogonal projection of $H_J$ to the orthogonal complement of $E_{n-1}$. Then, by Lemma~\ref{lem:Hu-ort-En-Hv} together with Lemma~\ref{lem:orthogonality-equivalence},
$$\widetilde H_{uI_n}\bigperp\widetilde H_{vI_n}$$
for every $n\in\mbN$ and $u\neq v$ in $\mbF$. Besides, $\tau_w\widetilde H_J=\widetilde H_{wJ}$ for every $w\in\mbF$ and $J\in\mbI$.

Similarly, we set:
$$\widetilde E_n = E_n \ominus E_{n-1}.$$
Therefore we have, by construction:
\begin{itemize}[leftmargin=20pt]
\item $H=\bigoplus_{n\in\mbN}\widetilde E_n$.
\item $\widetilde E_n=\bigoplus_{w\in\mbF}\widetilde H_{wI_n}$.
\item The action of $\mbF$ on $H$ permutes the summands inside each $\widetilde E_n$.
\end{itemize}
This implies that $\mbF\actson^\pi H$ is a multiple of the left-regular unitary representation of $\mbF$.
\end{proof}

Let us give a name to the phenomenon of the previous proposition.

\begin{defin}
Let $G$ be a topological group. A set $Q\subseteq G$ is a \emph{splitting Kazhdan set for $G$} if it generates a rank two free subgroup $F\leq G$ with the property that every unitary representation of $G$ with no invariant unit vectors restricts to a multiple of the left-regular representation of $F$.
\end{defin}

If our group $G$ has non-trivial unitary representations, then the group generated by $\set{\tau_a,\tau_b}$ will be a rank two free subgroup of $G$ (in fact, as per Corollary~\ref{cor:WAP-nonAP}, it suffices that the WAP compactification of $G$ be non-trivial). On the other hand, if $G$ does not have any non-trivial unitary representations, then any rank two free subgroup of $G$ (which exist, by Theorem~\ref{thm:AutQ-in-PRP}) is a splitting Kazhdan set for trivial reasons.

Thus, in all generality, we have proved the following.

\begin{theorem}\label{sT:main}
Every Roelcke precompact Polish group $G$ has Property (T). Moreover, if the Bohr compactification of $G$ is trivial but $G$ is not, then $G$ admits a splitting Kazhdan set.
\end{theorem}

\begin{cor}
A Roelcke precompact Polish group has strong Property~(T) if and only if its Bohr compactification does.
\end{cor}
\begin{proof}
The left-to-right implication is easy and holds in general; the converse follows from Corollary~\ref{cor:Bohr} and the moreover parts of the previous theorem and of Lemma~\ref{sT:lem:extensions}.
\end{proof}

\noindent\hrulefill

\section{Special cases}\label{sec:special-cases}

The construction of the action $\mbF\actson^\tau M$ of Section~\ref{sCai} involves a number of non-canonical choices. Under additional assumptions on the structure $M$, however, the construction can be rendered somewhat more canonical and simpler. As a result, in several concrete examples we are able to give a nice description of an action with the properties of Theorem~\ref{sCai:thm:cairns}, and therefore of a Kazhdan set for the automorphism group.

An instance of these arbitrary choices is the definition of the structure $M_{I_{n+1}}$, which is taken to be any submodel of the ambient model containing both $M_{I_n}$ and $M_{\ell_n I_n}$. In concrete examples there might exist a canonical submodel containing any such pair of models, for example the substructure generated by them or the real algebraic closure thereof.

A substantial simplification occurs if, in addition to this, the structure happens to be stable. In that case one can construct the action $\mbF\actson^\tau M$ in a direct manner, avoiding the inductive procedure of Proposition~\ref{sCai:prop:cairns}. Indeed, suppose the theory of $M$ is stable and the real algebraic closure of any two pair of submodels is again a submodel. Take any \emph{independent family} $\set{M_w}_{w\in\mbF}$ of copies of $M$ indexed by the free group. Recall that a family $\set{M_i}$ is independent if $M_{i_0}\ind M_{i_1}M_{i_2}\dots M_{i_n}$ for any distinct $i_0,\dots,i_n$. Because of our hypothesis, we may identify $M=\acl(\set{M_w}_{w\in\mbF})\cap\wM$, since this is indeed a separable submodel of~$\wM$. On the other hand, by the stationarity of independence in stable theories, we have $(M_w)_{w\in\mbF}\equiv (M_{aw})_{w\in\mbF}\equiv (M_{bw})_{w\in\mbF}$. Thus there are automorphisms $\tau_a$, $\tau_b$ of $M$ such that $\tau_a M_w = M_{aw}$ and $\tau_b M_w = M_{bw}$ for every $w\in\mbF$, inducing an action $\mbF\actson^\tau M$. Finally, given $I\in\mbI$, we can define $M_I\coloneqq\acl(\set{M_w}_{w\in I})\cap\wM$. It follows that the $M_I$ are submodels of $M$ satisfying $\tau_w M_I=M_{wI}$ and $M_I\ind_{M_K}M_J$ whenever $K=I\cap J$. So we have obtained an action $\mbF\actson^\tau M$ and a system $\set{M_I}_{I\in\mbI}$ with the properties of Theorem~\ref{sCai:thm:cairns}.

We summarize the previous discussion in the following theorem. We say that a subset $B\subseteq M$ \emph{algebraically generates} a substructure $N\subseteq M$ if $N=\acl(B)\cap M$.

\begin{theorem}\label{thm:special-cases}
Let $M$ be an $\aleph_0$-categorical, non-compact structure. Assume:
\begin{enumerate}
\item $M$ is stable.
\item The structure algebraically generated by any two elementary substructures of $M$ is an elementary substructure as well.
\item $M$ satisfies $\dcl(\emptyset)=\acl(\emptyset)$.
\end{enumerate}
Suppose moreover that $\set{M_w}_{w\in\mbF}$ is an independent family of elementary substructures of $M$ that algebraically generates $M$. Let $\tau_a,\tau_b\in\Aut(M)$ be automorphisms such that $\tau_a(M_w)=M_{aw}$ and $\tau_b(M_w)=M_{bw}$ for every $w\in\mbF$. Then $\set{\tau_a,\tau_b}$ is a splitting Kazhdan set for $\Aut(M)$.
\end{theorem}

Now we turn to some concrete examples.

\subsection*{The case of $\Unit(\ell^2)$} An explicit finite Kazhdan set for the unitary group $\Unit(\ell^2)$  was given in \cite{bekkaUnitary}: it consists of the generators of the unitary action $\mbF\actson\ell^2(\mbF)$. We can add to this that if we consider instead the unitary action $\mbF\actson \Unit(\ell^2(\mbF\times\mbN))$, then the generators of this action are moreover a \emph{splitting} Kazhdan set for the unitary group. Indeed, the underlying model (the Hilbert space) satisfies the hypotheses of Theorem~\ref{thm:special-cases}, and the presentation $M=\ell^2(\mbF\times\mbN)=\bigoplus_{w\in\mbF}\ell^2$ exhibits a generating, independent $\mbF$-family of submodels.

\subsection*{The cases of $S_\infty$ and $\GL(\infty,F_q)$} The same applies to the infinite symmetric group, $S_\infty$, and to the infinite-dimen\-sional general linear group over the finite field with $q$ elements, $\GL(\infty,F_q)$. Explicit finite Kazhdan sets for these oligomorphic groups were given in \cite{tsaUnitary}. They are induced by the permutation action $\mbF\actson\mbF$, after identifying $S_\infty$ with $\Sym(\mbF)$ and identifying $\GL(\infty,F_q)$ with the linear group of an $F_q$-vector space with base $\mbF$. As with the unitary group, we may consider instead the permutation action $\mbF\actson\mbF\times\mbN$ and the corresponding presentations of the groups $S_\infty$ and $\GL(\infty,F_q)$. By means of Theorem~\ref{thm:special-cases}, we get then splitting Kazhdan sets for these groups.

\subsection*{The case of $\Aut(\mu)$} The measure algebra of the unit interval $([0,1],\mu)$ with Lebesgue measure is also an $\aleph_0$-catego\-rical structure and enjoys the properties of the hypotheses of Theorem~\ref{thm:special-cases}; see \cite[\textsection 16]{bbhu08}. Up to isomorphism, we may identify it with the measure algebra $M$ of the space $([0,1]^\mbF,\mu^\mbF)$. The natural shift action $\mbF\actson [0,1]^\mbF$ induces then an action $\mbF\actson^\tau M$ of the kind stated in the theorem. Indeed, for each $w\in\mbF$, the submodel $M_w$ corresponds to the algebra of sets that are measurable with respect to the $w$-th coordinate. By identifying $\Aut(\mu)\simeq\Aut([0,1]^\mbF,\mu^\mbF)\simeq\Aut(M)$, we obtain a splitting Kazhdan set for the group of measure-preserving transformations of the interval.

\subsection*{The case of $\Aut^*(\mu)$} The group $\Aut^*(\mu)$ of non-singular transformations of $([0,1],\mu)$ (i.e., transformations that preserve null sets) can also be seen as the automorphism group of an $\aleph_0$-catego\-rical structure. Indeed, $\Aut^*(\mu)\simeq\Aut(M)$ where $M$ is the \emph{Banach lattice} $L^1([0,1])$, i.e., the Banach space $L^1([0,1])$ augmented with operations for the pointwise maximum and minimum of pairs of $L^1$ functions. In fact, if $\Omega\subseteq\mbR$ is any measurable subset of positive Lebesgue measure, then $M$ and $L^1(\Omega)$ are isomorphic as Banach lattices, and the natural inclusion $L^1(\Omega)\subseteq L^1(\mbR)$ is elementary. Here as well, the structure $M$ satisfies all the hypotheses of Theorem~\ref{thm:special-cases}; see \cite[\textsection 17]{bbhu08} and \cite{bbhLplattice}.

Now let $w\in\mbF\mapsto n_w\in\mbN$ be any bijection, and define $\Omega_w=[n_w,n_w+1)\subseteq\mbR$ and $M_w=L^1(\Omega_w)$. Then $\set{M_w}_{w\in\mbF}$ is an independent, generating family of submodels of $L^1(\mbR)$. If $\tau_a,\tau_b\colon\mbR\to\mbR$ are, for instance, piecewise translations of $\mbR$ such that $\tau_a(\Omega_w)=\Omega_{aw}$ and $\tau_b(\Omega_w)=\Omega_{bw}$ for every $w\in\mbF$, then they induce automorphisms of $L^1(\mbR)$ with $\tau_a (M_w) = M_{aw}$ and $\tau_b (M_w) = M_{bw}$. We conclude that $\set{\tau_a,\tau_b}$ is a splitting Kazhdan set for $\Aut(L^1(\mbR))\simeq\Aut^*(\mu)$.

\subsection*{The case of randomized groups} Let $\Omega=[0,1]$. If $G$ is a Polish group, we can consider the group $L^0(\Omega,G)$ of measurable maps $\Omega\to G$ up to $\mu$-almost everywhere equality, endowed with the topology of convergence in measure, which is a Polish group as well; see the book \cite[\textsection 19]{kechrisGlobal}. As mentioned in the introduction, Pestov \cite{pestovVersus} observed that this group need not have Property (T) even if $G$ is compact.

On the other hand, we may form the semi-direct product $G\wr\mu\coloneqq L^0(\Omega,G)\rtimes\Aut(\Omega,\mu)$. If $G$ is the automorphism group of a separable metric structure $M$, then $G\wr\mu$ is the automorphism group of the \emph{Borel randomization} of $M$, as shown in \cite{ibaRando}. The passage from $M$ to its randomization preserves $\aleph_0$-categoricity. Hence, if $G$ is Roelcke precompact (in particular, if $G$ is compact), then so is $G\wr\mu$, and thus $G\wr\mu$ has Property (T). Moreover, the randomization always satisfies $\dcl(\emptyset)=\acl(\emptyset)$ (see Ben Yaacov's \cite[Corollary~5.9]{benOntheories}), so $G\wr\mu$ has strong Property~(T).

It might be worth noting that a Kazhdan set for $G\wr\mu$ cannot in general be contained in the subgroup $\Aut(\Omega,\mu)$. For example, if $G=\set{-1,1}$ is the group with two elements, the unitary representation $G\wr\mu\actson^\sigma L^2(\Omega)$ given by $(\sigma_{(g,t)}\xi)(\omega)=g(\omega)\xi(t\inv\omega)$ for each $g\in L^0(\Omega,G)$ and $t\in\Aut(\Omega,\mu)$ has no invariant unit vectors, yet the subgroup $\Aut(\Omega,\mu)$ fixes the constant functions of $L^2(\Omega)$. A similar argument works for any $G$ admitting non-trivial unitary representations.

Given that randomizations preserve stability, it should be possible to exhibit Kazhdan sets for many concrete cases of groups of the form $G\wr\mu$ following the strategy discussed in this section. In particular, one should be able to describe a finite Kazhdan set for $G\wr\mu$ when $G$ is a compact metrizable group. However, to implement our method one needs a description of the stable independence relation in randomizations that is not directly available in the literature, and which is not our aim to develop here.

\noindent\hrulefill

\bibliographystyle{amsalpha}
\bibliography{biblio}

\end{document}